\documentclass{amsart}
\usepackage[T1]{fontenc}
\usepackage{amssymb, amsmath, color, textcomp, url,tikz, mathtools}
\usepackage{bbm}
\usepackage[noadjust]{cite}
\usepackage{tikz-cd}

\usetikzlibrary{matrix,arrows}
\usetikzlibrary{fit}
\usetikzlibrary{positioning}
\usepackage[labelformat=empty]{caption}

\usepackage[normalem]{ulem}
\usepackage{soul}
\usepackage{mdwlist}

\usepackage[pdftex]{hyperref}

\usepackage{enumerate,xspace}
\usepackage{chngcntr, csquotes} 
\theoremstyle{plain}
\newtheorem{theorem}{Theorem}[section]
\newtheorem{cor}[theorem]{Corollary}
\newtheorem{prop}[theorem]{Proposition}
\newtheorem{lemma}[theorem]{Lemma}

\newcounter{proofcount}
\AtBeginEnvironment{proof}{\stepcounter{proofcount}}

\newtheoremstyle{claimstyle} 
{}                  
{}                  
{\itshape}                 
{}                         
{\itshape}        
{.}                        
{ }                        
{}                         

\theoremstyle{claimstyle}
\newtheorem{claim}{Claim}

\newtheorem*{claim*}{Claim}
\makeatletter                  
\@addtoreset{claim}{proofcount}
\makeatother                   

\newenvironment{claimproof}[1][Proof of Claim \theclaim.] 
{%
	\proof[#1]%
	
}
{%
	\endproof%
}

\newenvironment{claimproof*}[1][Proof of Claim.] 
{%
	\proof[#1]%
	
}
{%
	\endproof%
}

\theoremstyle{theorem}

\newtheorem*{teoA}{Theorem A}
\newtheorem*{teoB}{Theorem B}
\newtheorem*{teoC}{Theorem C}

\theoremstyle{definition}
\newtheorem{remark}[theorem]{Remark}
\newtheorem{fact}[theorem]{Fact}
\newtheorem{definition}[theorem]{Definition}
\newtheorem{example}[theorem]{Example}

\newtheorem*{notation}{Notation}

\newcommand{\nc}{\newcommand}

\nc{\ZB}{\mathbb{Z}}
\nc{\QB}{\mathbb{Q}}
\nc{\NB}{\mathbb{N}}
\nc{\FB}{\mathbb{F}}
\nc{\UB}{\mathbb{U}}
\nc{\MB}{\mathbb{M}}
\nc{\CB}{\mathbb{C}}
\nc{\RB}{\mathbb{R}}
\nc{\ideal}{I}
\nc{\power}[1]{\mcl{P}(#1)}
\nc{\types}[1]{S(#1)}
\nc{\ftypes}[1]{S_{\varphi}(#1)}
\nc{\gl}[1]{\mathfrak{#1}}
\nc{\rst}[1]{\vert_{#1}}
\nc{\goo}[1]{G_{#1}^{00}}
\nc{\deftp}[2]{\tilde{\textrm{d}}_{#1} #2}
\nc{\deftpth}[2]{\tilde{\emph{d}}_{#1} #2}
\nc\inv{^{-1}}
\nc{\stab}{\operatorname{Stab}}
\nc{\dcl}{\operatorname{dcl}}
\nc{\acl}{\operatorname{acl}}
\nc{\tpp}{\operatorname{tp}}
\nc{\abs}[1]{\vert#1\vert}
\nc{\Aut}{\operatorname{Aut}}
\nc{\mcl}[1]{\mathcal{#1}}
\nc{\eps}{\varepsilon}
\nc{\dif}[1]{\operatorname{d}#1}
\nc{\wide}[1]{\textrm{w}_{#1}}

\def\Ind#1#2{#1\setbox0=\hbox{$#1x$}\kern\wd0\hbox to 0pt{\hss$#1\mid$\hss}
	\lower.9\ht0\hbox to 0pt{\hss$#1\smile$\hss}\kern\wd0}

\def\Notind#1#2{#1\setbox0=\hbox{$#1x$}\kern\wd0\hbox to 0pt{\mathchardef
		\nn="3236\hss$#1\nn$\kern1.4\wd0\hss}\hbox to 0pt{\hss$#1\mid$\hss}\lower.9\ht0
	\hbox to 0pt{\hss$#1\smile$\hss}\kern\wd0}

\begin{document}

	\title{On the regularity of almost stable relations}
	\date{\today}
	
	\author{Marcos Gir\'on}
	\address{Departamento de \'Algebra, Geometr\'ia y Topolog\'ia, Facultad de Matem\'aticas,
		Universidad Complutense de Madrid, 28040 Madrid, Spain}
	\email{marcosgi@ucm.es}
	
	
	\begin{abstract}
		We develop a general theory of local stability up to belonging to an ideal (e.g. having measure zero). From a model-theoretic perspective, we prove a stationarity principle for almost stable formulas in this sense, and build a topological space of partial types whose Cantor-Bendixson rank is finite. The interaction of this space with Keisler measures and definable groups yields, on the one hand, a regularity lemma for infinite graphs where the edge relation is almost stable, and, on the other hand, the existence of definable stabilizer subgroups. As an application, we prove a finite graph regularity lemma and an arithmetic regularity lemma for almost stable relations in arbitrary finite groups.
	\end{abstract}
	
	\maketitle 
	
	\section{Introduction}
	
	Szemer\'edi regularity lemma \cite{SZ74, SZ78} is a fundamental structural theorem that has had many applications in graph theory and combinatorics, and motivated the celebrated arithmetic regularity results of Green on abelian groups \cite{BG05}. Recently, model theory has been used to obtain strengthened regularity lemmas for graphs and for subsets of groups omitting half-graphs of some fixed size.
	
	More precisely, a subset $A$ of a group $G$ is \emph{$k$-stable} if it induces no half-graph of height $k$: there do not exist $g_1,\ldots,g_k,h_1,\ldots,h_k\in G$ such that $g_i^{-1}\cdot h_j\in A$ if and only if $i\leq j$. In this context, is of particular importance the work of Terry and Wolf on groups of the form $\mathbb{F}_p^n$ \cite{TW19} and on abelian groups \cite{TW20}, as well as the results of Conant, Pillay and Terry concerning arbitrary finite groups \cite{CPT18} (see also \cite{C21,ADJ21}). Essentially, all these works establish that a stable subset of a group is approximately a union of cosets of a normal subgroup of bounded index. Moreover, this line of research has already been extended beyond the stable context, and results of this nature have been successfully extended for subsets of groups with finite VC-dimension \cite{AFZ19, CPT22, CP24, CT25}.
	
	Another remarkable advance of direct relevance to our purposes was achieved by Terry and Wolf in \cite{TW23}. Among other results, the authors prove an arithmetic regularity lemma for \emph{almost $k$-stable subsets} of finite abelian groups, which can be intuitively understood as sets inducing very few half-graphs of height $k$. Using techniques from generalized stability theory, we extend this theorem to arbitrary finite groups. To make precise our statement, we introduce the following notation: given a binary relation $E$ in a set $M$, let
	\begin{align*}
		\mcl{H}_k(E)=\{(a_1,b_1,\ldots,a_k,b_k)\in M^{2k}:(a_i,b_j)\in E\Leftrightarrow i\leq j\}
	\end{align*}
	denote the set of half-graphs of height $k$ induced by $E$ in $M$. Then, if we bound the cardinality of $\mcl{H}_k(E)$ in the particular case where $E$ is the Cayley relation induced by a subset $A$ of a group $G$, namely
	\begin{align*}
		\textrm{Cay}(G,A)=\{(g,h)\in G\times G:g^{-1}\cdot h\in A\},
	\end{align*}
	we obtain the following theorem:
	
	\begin{teoA}[Theorem \ref{fgroups}]
		For every natural number $k$ and every real number $\eps>0$, there exists a natural number $n=n(k,\eps)$ such that: for every finite group $G$ and every subset $A$ of $G$, if $\abs{\mcl{H}_k(\emph{Cay}(G,A))}<\abs{G}^{2k}/n$, then there exists a normal subgroup $H$ of $G$ with index at most $n$, and a union of cosets $B$ of $H$ such that $\abs{A\triangle B}<\eps\cdot\abs{H}$.
	\end{teoA}
	
	Theorem A is the corresponding version of \cite[Theorem 1.3]{CPT18} for almost stable sets, where the latter extends the regularity lemma of Terry and Wolf on groups of the form $\mathbb{F}_p^n$ to arbitrary finite groups. For generalizing \cite[Theorem 5.25]{TW23} to arbitrary finite groups in Theorem A, we need to develop the model theory of almost stable relations. Our work will also allow us to analyse regularity phenomena in the context of finite graphs, which was originally the first result to be studied model theoretically.
	
	Recall that a bipartite graph $(U,V,E)$ is \emph{$k$-stable} if the edge relation $E$ induces no half-graph of height $k$: there are no vertices $a_1,\ldots,a_k\in U$ and $b_1,\ldots,b_k\in V$ such that $(a_i,b_j)\in E$ if and only if $i\leq j$. In this case, in their seminal work \cite{SM14}, Malliaris and Shelah proved that we can obtain partitions $U_1\cup\ldots\cup U_r$ of $U$ and $V_1\cup\ldots\cup V_s$ of $V$ such that $E\vert_{U_i\times V_j}$ is $\eps$-\emph{homogeneous} for \emph{all} pairs $(i,j)$, that is, for every $(i,j)$ either
	\begin{align*}
		\abs{E\cap(U_i\times V_j)}\geq(1-\eps)^2\cdot\abs{U_i\times V_j}\textrm{ or }\abs{\neg E\cap(U_i\times V_j)}\geq(1-\eps)^2\cdot\abs{U_i\times V_j}.
	\end{align*}
	A pseudofinite proof of this theorem using local stability theory, as developed by Hrushovski and Pillay \cite{HP94}, was given later by Malliaris and Pillay in \cite{PM15} (afterwards, this proof was simplified by Pillay in \cite{AP20}). We also refer to \cite[Theorem 2.1.2]{CT23}, where Conant and Terry give a sharper qualitative proof for an error function rather than a constant $\eps>0$ fixed beforehand, at the cost of allowing some bounded number of exceptional pairs. Furthermore, similar results have been obtained in other tame contexts aligning with classical classification lines in model theory, such as graphs of bounded VC-dimension or graphs definable in distal structures, see \cite{AFN07,CS18, CS21,LS10}.
	
	A notable improvement in this framework has also recently been achieved by Terry and Wolf in \cite[Lemma 6.14]{TW21}. Their result constitutes a regularity lemma for \emph{almost stable graphs}, or graphs whose edge relation has very few witnesses to instability. Although no explicit bound appears in the statement, their proof yields polynomial bounds on the sizes of the resulting partitions. Thus, from \cite[Lemma 6.14]{TW21} one can deduce the following lemma for bipartite graphs:
	\begin{teoB}[Theorem \ref{fgraphs}]
		For every natural number $k$ and all real numbers $\eps,\delta>0$, there are natural numbers $n=n(k,\eps,\delta)$, $r=r(k,\eps,\delta)$ and $s=s(k,\eps,\delta)$ such that: if $G=(U,V,E)$ is any finite bipartite graph with $\abs{U},\abs{V}\geq n$ and $\abs{\mcl{H}_k(E)}<\abs{U}^{k}\abs{V}^k/n$, then we can partition $U$ into sets $U_0,U_1,\ldots,U_r$ and $V$ into sets $V_0,V_1,\ldots,V_s$ with $\abs{U_i}\geq\abs{U}/n$ and $\abs{V_i}\geq\abs{V}/n$ for every $i\geq 1$, such that
		\begin{align*}
			\abs{U_0}<\delta\cdot\abs{U} \textrm{ and } \abs{V_0}<\delta\cdot\abs{V}
		\end{align*}
		and for every $i,j\geq 1$, the relation $E\vert_{U_i\times V_j}$ is $\eps$-homogeneous.
	\end{teoB}
	Following Malliaris and Pillay \cite{PM15}, we use model-theoretic techniques to refine this theorem by showing that the sets $U_1,\ldots,U_r$ can be chosen to be Boolean combinations of neighbourhoods of the form $E(x,b)$ and $\neg E(x,b)$, and likewise for $V_1,\ldots,V_s$. In addition, the complexity of these Boolean combinations is uniform on $k$, $\eps$ and $\delta$ (see Theorem \ref{regularmeasure} and Theorem \ref{fgraphs}). Note that our result resembles \cite[Theorem 5.6]{AFP17}, an important precedent in the literature which provided a similar regularity lemma in finite graphs. On the other hand, as a corollary of Theorem \ref{regularmeasure}, we answer a question of \cite[Page 14]{AMP242} by giving a model-theoretic account of \cite[Theorem 6.13]{TW21} (see also \cite{EFR86}).
	
	As noted above, in this article we develop a framework in which almost stable relations as described earlier can be systematically studied from a model-theoretic perspective, with the additional aim of contributing to the transition from stability to almost stability which is already in progress in the aforementioned literature.

	
	Indeed, over the past few years, several works have relaxed the notion of stability to \emph{almost stability} allowing model-theoretic techniques from local stability theory to be applied beyond the stable context \cite{AFP17,CT24,CM22,AMP242,TW23}. The central notion in all these works is what we call here \emph{almost sure stable relations}: in the presence of a measure $\mu$, a relation $E$ is almost sure stable if, for some natural number $k$, we have that $\mu(\mcl{H}_k(E))=0$. This weakening of stability is robust from a combinatorial perspective as it is insensitive to perturbations by sets of measure zero. We generalize almost sure stable relations to $\ideal$\emph{-stable formulas} (see Definition \ref{Istabformula}), where $\ideal$ is an $S_1$-ideal of definable sets, as introduced by Hrushovski \cite{EH12} (see Definition \ref{S1}). When $\ideal$ is the family of sets of measure zero with respect to a Keisler measure (that is, a finitely additive probability measure on the Boolean algebra of definable sets), we recover the setting of almost sure stable relations (see Fact \ref{widetuple}).
	
	\medskip
	We now proceed to outline the structure of the article, emphasizing the model-theoretic component, which constitutes the major novelty of our work. We aim to develop a general theory of local stability up to belonging to an $S_1$-ideal $\ideal$. To this end, we define in Section \ref{prel} the notion of independence with respect to a family of ideals (Definition \ref{wide&indep}) and we relate it with the usual non-forking independence. This allows us to prove in Section \ref{Istab}, following the lines of \cite[Theorem 4.11]{AMP242}, a stationarity principle for $\ideal$-stable formulas (Theorem \ref{stat}), as well as a definability result for wide types (Theorem \ref{def2}). Recall that a type is wide if it avoids all the formulas from the ideal $\ideal$. Our main model-theoretic result is:
	
	\begin{teoC} [Theorem \ref{finiteCB}]
		Let $\ideal$ be an $\emptyset$-invariant $S_1$-ideal and let $\varphi(x,y)$ be an $\ideal$-stable formula. Given a model $M$, let $\gl{X}_{\varphi,M}$ be the set whose elements are maximal consistent sets of Boolean combinations of formulas of the form $\varphi(x,b)$ and $\neg\varphi(x,b)$, with  $\tpp(b/M)$ wide. Then $\gl{X}_{\varphi,M}$ has a natural structure of compact, Hausdorff and $0$-dimensional topological space with finite Cantor-Bendixson rank.
	\end{teoC}
	
	This theorem, as well as the model-theoretic machinery developed in Section \ref{Istab}, plays a crucial role in sections \ref{regularity} and \ref{groups}, where we prove two topological regularity lemmas in the $\ideal$-stable context, which we will now proceed to specify. These are applied in Section \ref{applications} to obtain Theorem A and Theorem B.
	
	In Section \ref{regularity} we work with Keisler measures. In this context, we can obtain \emph{regular partitions} of $\gl{X}_{\varphi,M}$ leading to a regularity lemma for infinite graphs where the edge relation is $\ideal$-stable (Theorem \ref{regularmeasure}), generalizing the main result of Malliaris and Pillay \cite{PM15}. In Section \ref{groups} we study groups and stabilizers. The presence of an ambient definable group $G$ acting on $\gl{X}_{\varphi,M}$ implies the existence of definable stabilizer subgroups of $G$. As a consequence, we obtain a regularity lemma for infinite groups: for every $\ideal$-stable subset $A$ of $G$ (that is, the set $A$ is definable and the formula $y\cdot x\in A$ is $\ideal$-stable) there exists a finite union $B$ of cosets of a normal finite-index subgroup of $G$ such that $A\triangle B\in\ideal$ (Theorem \ref{almoststable}).
	
	Finally, in Section \ref{applications} we apply these results to the case of finite graphs and arbitrary finite groups to prove Theorem A and Theorem B via an standard ultraproduct argument. 
	
	\medskip
	
	\noindent\textbf{Acknowledgements.} The author is supported by Spanish STRANO PID2021-122752NB-I00 and a fellowship (grant CT15/23) funded by Universidad Complutense de Madrid and Banco Santander. He thanks Julia Wolf for her useful insights in relation to the appendix of this article, as well as the anonymous referee for their suggestions and comments regarding important theorems preceding our work. The results of this paper are part of the author’s Ph.D. dissertation, which is supervised by El\'ias Baro and Daniel Palac\'in.

	\section{Preliminaries}\label{prel}
	
	Our notation is standard and we assume that the readers are familiar with the basics of local stability theory. We refer to the first chapter of \cite{AP96} or \cite{DP18}.
	
	Throughout the article, we fix a complete first-order theory $T$, possibly multi-sorted, in a language $L$. We assume that $T$ has models all whose sorts are infinite. As usual, we will be working inside a monster model $\MB$ of the theory and we denote by $x,y,z,\ldots$ finite tuples of variables.
	
	Recall that a family $\ideal$ of $L_x(\MB)$-formulas is an \textit{ideal} if it is closed under subsets and finite unions, and additionally $\emptyset\in\ideal$. We write $\ideal_x$ to stress that the formulas are in the fixed finite tuple of variables $x$. 
	
	Let $\ideal_x$ be an ideal. Following Hrushovski \cite{EH12}, we say that a partial type $\pi(x)$ is $\ideal_x$\emph{-wide} if it implies no formula from $\ideal_x$. A tuple $a$ is $\ideal_x$\emph{-wide over} $A$ if the type $p(x)=\tpp(a/A)$ is $\ideal_x$-wide.
	
	\begin{definition}\label{wide&indep}
		Let $(x_1,\ldots,x_n)$ be a finite tuple and let $\ideal_{x_i}$ be an ideal in each variable $x_i$. A finite sequence $(a_1,\ldots,a_n)$ is said to be $A$\textit{-independent} or \textit{independent over} $A$ \emph{with respect to the ideals} $\ideal_{x_i}$ if $\tpp\left(a_i/A,a_{<i}\right)$ is $I_{x_i}$-wide for every $i=1,\ldots,n$, where $a_{<i}=\{a_j:j<i\}$. We say that a finite sequence is \textit{non-forking} $A$\textit{-independent} if it is $A$-independent with respect to the ideal of formulas that fork over $A$.
	\end{definition}
	
	When the ideals $\ideal_{x_i}$ are understood from the context, we will simply say that a tuple is $A$-independent or independent over $A$.
	
	A basic property that we will use repeatedly is that wide types can be extended to wide types over larger sets of parameters. More precisely, if $\pi(x)$ is a partial type over $A$ that is wide with respect to an ideal $\ideal_x$, then for every set $B\supset A$ there exists a complete $\ideal_x$-wide type $p(x)$ over $B$ extending $\pi(x)$. The reason is that, by compactness, the set of formulas
	\begin{align*}
		\pi(x)\cup\{\neg\varphi(x,b):b\in B\textrm{ and }\varphi(x,b)\in \ideal_x\}
	\end{align*}
	is consistent, and hence it can be extended to a complete type over $B$. This type is necessarily $\ideal_x$-wide. In particular, a type-definable set $X$ is $\ideal_x$-wide (that is, it is defined by an $\ideal_x$-wide partial type) if and only if it contains an element that is $\ideal_x$-wide over the parameters of definition of $X$. In that case, for every set $B$ containing such parameters, the type-definable set $X$ contains an element that is $\ideal_x$-wide over $B$.
	
	As usual, an ideal $\ideal_x$ is $A$-invariant if it is invariant under the set of automorphisms $\Aut_A(\MB)$ fixing pointwise the set $A$. Equivalently, if $a\equiv_Aa^\prime$ and $\varphi(x,a)\in\ideal_x$, then $\varphi(x,a^\prime)\in\ideal_x$.

	\begin{definition}\label{S1}
		\cite[Definition 2.8]{EH12} An $A$-invariant ideal $\ideal_x$ is $S_1$ if the following holds: for every formula $\varphi(x,y)\in L$ and every $A$-indiscernible sequence $(a_i)_{i\in\NB}$ such that $\varphi(x,a_0)\land\varphi(x,a_1)\in\ideal_x$, we have that $\varphi(x,a_0)\in\ideal_x$.
	\end{definition}
	
	As noted in \cite[Lemma 2.9]{EH12}, if the ideal $\ideal_x$ is $A$-invariant and $S_1$, and $\pi(x)$ is an $\ideal_x$-wide partial type, then $\pi(x)$ does not fork over $A$. In particular, every $A$-independent tuple $(a_1,\ldots,a_n)$ is non-forking $A$-independent.
	
	\begin{example}\label{measures}
		Let $\mu$ be a \textit{global Keisler measure}, that is, a finitely additive probability measure on the Boolean algebra of definable sets over $\MB$, or equivalently, on $L_x(\MB)$. Then the collection of sets of measure zero is an ideal, say $\ideal_x$.
		
		We shall say that the measure $\mu$ is $A$\emph{-invariant} (see \cite[Definition 2.13]{KG20}) if for every formula $\varphi(x,y)\in L$ and all $b,b^\prime\in\MB^{y}$ such that $b\equiv_Ab^\prime$ we have that $\mu(\varphi(x,b))=\mu(\varphi(x,b^\prime))$. If $\mu$ is $A$-invariant then the ideal $\ideal_x$ is $A$-invariant and $S_1$ \cite[Example 2.12]{EH12}, and formulas of positive measure do not fork over $A$.
	\end{example}
	
	The property $S_1$ of ideals is closely related with the notion of \textit{equational relation}. An $A$\emph{-invariant} relation $R(x,y)$ is just a subset of $\MB^{x}\times\MB^{y}$ such that for every $(a,b)\equiv_A(a^\prime,b^\prime)$, if $R(a,b)$ holds then so does $R(a^\prime,b^\prime)$. We say that an $A$-invariant relation is \textit{equational} if there is no $A$-indiscernible sequence $(a_i,b_i)_{i\in\NB}$ such that $R(a_0,b_1)$ holds but $R(a_0,b_0)$ does not. 
	
	\begin{fact}\label{S1eq}\cite[Lemma 2.10]{EH12}
		Let $\ideal_x$ be an $A$-invariant $S_1$-ideal, and let $\Phi(x,y)$ and $\Psi(x,z)$ be partial types over $A$. The relation $R_{\Phi,\Psi}$ defined as
		\begin{align*}
			R_{\Phi,\Psi}(a,b)\Leftrightarrow \Phi(x,a)\cup\Psi(x,b)\textrm{ is not } I_x\textrm{-wide}
		\end{align*}
		is $A$-invariant and equational.
	\end{fact}
	
	\begin{fact}\label{stateq}\cite[Remark 2.1]{AMP243}
		Let $M$ be a model, and let $R(x,y)$ be an $M$-invariant equational relation. Suppose that $R(a,b)$ holds for some $a$ and $b$ such that $\tpp(a/M,b)$ or $\tpp(b/M,a)$ does not divide over $M$. Then, for any $a^\prime\equiv_Ma$ and $b^\prime\equiv_Mb$, we have that $R(a^\prime,b^\prime)$ holds.
	\end{fact}	
	
	We now prove the main result of this first section (cf. \cite[Fact 4.5]{AMP242}):
	
	\begin{lemma}\label{wideinter}
		Let $M$ be a model, let $I_x$ an $M$-invariant $S_1$-ideal, and let $\Phi(x,y)$ a partial type over $M$. If $\Phi(x,a)$ is $I_x$-wide, then $\bigcup_{i=1}^n\Phi(x,a_i)$ is also $I_x$-wide, for every $n\in\NB$ and every non-forking $M$-independent tuple $(a_1,\ldots,a_n)$ of realizations of  $\tpp(a/M)$.
	\end{lemma}
	
	\begin{proof}
		We will prove it by induction on $n\geq 2$. For $n=2$, this is the content of \cite[Corollary 2.3]{AMP243}. Suppose that both $b$ and $c$ realize $\tpp(a/M)$ and that the type $\tpp(c/M,b)$ does not fork over $M$. If $\Phi(x,b)\cup\Phi(x,c)$ is not $I_x$-wide, then $R_{\Phi,\Phi}(b,c)$ holds. This implies, by Facts \ref{S1eq} and \ref{stateq}, that $R_{\Phi,\Phi}(a,a)$ also holds, which is a contradiction, as $\Phi(x,a)$ is wide by hypothesis.
		
		So, by induction, consider a non-forking $M$-independent tuple $(a_1,\ldots,a_n)$ with $a_i\equiv_Ma$ for every $i=1,\ldots,n$. First, we claim that:
		
		\begin{claim*}
			The tuple $(a_1,\ldots,a_n)$ can be extended to a non-forking $M$-independent tuple $(a_1,\ldots,a_n,a_{n+1},\ldots,a_{2n-2})$ such that $(a_1,\ldots,a_{n-1})\equiv_M(a_n,\ldots,a_{2n-2})$.
		\end{claim*}
		
		\begin{claimproof*}
			For every $j=0,\ldots,n-2$, we will recursively define $a_{n+j}$ so that $(a_1,\ldots,a_{n+j})$ is non-forking $M$-independent and $(a_1,\ldots,a_{j+1})\equiv_M(a_n,\ldots,a_{n+j})$. The case $j=n-2$ is precisely the claim.
			
			Suppose that $a_{n+j}$ has already been constructed. By hypothesis, there exists an automorphism $f\in\Aut_M(\MB)$ such that $f(a_1,\ldots,a_{j+1})=(a_n,\ldots,a_{n+j})$. The type $\tpp(a_{j+2}/M,a_1,\ldots,a_{j+1})$ does not fork over $M$, and neither does the type $\tpp(f(a_{j+2})/M,a_n,\ldots,a_{n+j})$. Hence, this latter type can be extended to a complete type $p(x)$ over $M,a_1,\ldots,a_n,\ldots,a_{n+j}$ which does not fork over $M$. We can define $a_{n+j+1}$ as a realization of $p(x)$.
		\end{claimproof*}
		
		By left-transitivity of non-forking, the type $\tpp(a_n,\ldots,a_{2n-2}/M,a_1,\ldots,a_{n-1})$ does not fork over $M$. Therefore, since by induction hypothesis both partial types $\Phi(x,a_1)\cup\ldots\cup\Phi(x,a_{n-1})$ and $\Phi(x,a_n)\cup\ldots\cup\Phi(x,a_{2n-2})$ are $\ideal_x$-wide, we can apply the case $n=2$ to conclude that
		\begin{align*}
			\left(\Phi(x,a_1)\cup\ldots\cup\Phi(x,a_{n-1})\right)\cup\left(\Phi(x,a_n)\cup\ldots\cup\Phi(x,a_{2n-2})\right)
		\end{align*}
		is $\ideal_x$-wide. In particular $\Phi(x,a_1)\cup\ldots\cup\Phi(x,a_n)$ is $\ideal_x$-wide, as we wanted to prove.
	\end{proof}

	\section{Almost stable formulas}\label{Istab}
	
	Fix an ideal $\ideal_x$ in $L_x(\MB)$ and an ideal $\ideal_y$ in $L_y(\MB)$, and denote by $\ideal$ the ordered pair $(\ideal_x,\ideal_y)$. We assume that both $\ideal_x$ and $\ideal_y$ are $\emptyset$-invariant and $S_1$.
	
	\begin{definition}\label{Istabformula}
		A formula $\varphi(x,y)\in L$ is $\ideal$\emph{-stable of ladder} $k$ if there is no model $M$ and no tuple $(a_1,b_1,\ldots,a_k,b_k)$ such that either $(a_1,b_1,\ldots,a_k,b_k)$ or $(b_k,a_k,\ldots,b_1,a_1)$ is $M$-independent with respect to $\ideal_x$ and $\ideal_y$, and $\varphi(a_i,b_j)$ holds if and only if $i\leq j$. We say that $\varphi(x,y)$ is $I$\emph{-stable} if it is $I$-stable of ladder $k$ for some $k\in\NB$.
	\end{definition}
	
	In Corollary \ref{permutation} we will see that the above definition can be reformulated as: $\varphi(x,y)$ is $\ideal$-stable of ladder $k$ if there is no model $M$ and no tuple $(a_1,b_1,\ldots,a_k,b_k)$ such that some permutation in $S_{2k}$ of $(a_1,b_1,\ldots,a_k,b_k)$ is $M$-independent, and $\varphi(a_i,b_j)$ holds if and only if $i\leq j$.
	
	\begin{lemma}\label{comb}
		Let $\varphi(x,y)$ be an $\ideal$-stable formula. Then so are the opposite formula $\varphi^*(y,x)$ and the negation $\neg\varphi(x,y)$.
	\end{lemma}
	
	\begin{proof}
		Assume that $\varphi(x,y)$ is $\ideal$-stable of ladder $k$.
		
		The opposite formula $\varphi^*(y,x)$ is $\ideal$-stable of ladder $k$. Otherwise, there is some model $M$ and some sequence $(b_1,a_1,\ldots,b_k,a_k)$ such that either $(b_1,a_1,\ldots,b_k,a_k)$ or $(a_k,b_k,\ldots,a_1,b_1)$ is $M$-independent, and with the property that $\varphi^*(b_i,a_j)$ holds if and only if $i\leq j$. Set $c_i=a_{k-i+1}$ and $d_i=b_{k-i+1}$ for $i=1,\ldots,k$. By construction, either $(d_k,c_k,\ldots,d_1,c_1)$ or $(c_1,d_1,\ldots,c_k,d_k)$ is $M$-independent, and $\varphi(c_i,d_j)$ holds if and only if $i\leq j$, which is a contradiction.
		
		Similarly, $\neg\varphi(x,y)$ is $\ideal$-stable of ladder $k+1$. To prove it, suppose to the contrary that there is some model $M$ and some tuple $(a_1,b_1,\ldots,a_{k+1},b_{k+1})$ such that either $(a_1,b_1,\ldots,a_{k+1},b_{k+1})$ or $(b_{k+1},a_{k+1},\ldots,b_1,a_1)$ is $M$-independent, witnessing that $\neg\varphi(x,y)$ is not $\ideal$-stable of ladder $k+1$. Define $c_i=a_{i+1}$ and $d_i=b_i$ for every $i=1,\ldots,k$. Then either $(d_1,c_1,\ldots,d_k,c_k)$ or $(c_k,d_k,\ldots,c_1,d_1)$ is $M$-independent, and $\varphi^*(d_i,c_j)$ holds if and only if $i\leq j$, contradicting the above paragraph.
	\end{proof}

	A key fact about $\ideal$-stable formulas is that their truth value depends only on the type over some fixed model of each of the coordinates, provided that the pairs of realizations are independent over such model. In the case of stable formulas and non-forking independence, this is a classical result \cite[Lemma 2.3]{EH12}.

	\begin{theorem}[Stationarity]\label{stat}
		Let $\varphi(x,y)$ be an $\ideal$-stable formula and let $M$ be a model. Assume that $(a,b)$ or $(b,a)$ is $M$-independent. Then for every $a^\prime$ and $b^\prime$ such that $a\equiv_Ma^\prime$, $b\equiv_Mb^\prime$, and $(a^\prime,b^\prime)$ or $(b^\prime,a^\prime)$ is $M$-independent, we have that $\varphi(a,b)$ holds if and only if $\varphi(a^\prime,b^\prime)$ does.
	\end{theorem}
	
	\begin{proof}
		We follow the lines of \cite[Theorem 4.11]{AMP242}. By Lemma \ref{comb}, the formulas $\neg\varphi(x,y)$, $\varphi^*(y,x)$ and $\neg\varphi^*(y,x)$ are $\ideal$-stable, so we may assume that $\varphi(a,b)$ holds and that $(a,b)$ is $M$-independent. We will distinguish two cases: $(b^\prime,a^\prime)$ is $M$-independent or $(a^\prime,b^\prime)$ is $M$-independent.
		
		Suppose first that $(b^\prime,a^\prime)$ is $M$-independent. Without loss of generality, we can assume that $b^\prime=b$. If $\varphi(a^\prime,b)$ does not hold, then we have two pairs $(a,b)$ and $(a^\prime,b)$ satisfying the following properties:
		\begin{enumerate}
			\item[$(1)$] the type $\tpp(a/M)=\tpp(a^\prime/M)$ is $\ideal_x$-wide,
			\item[$(2)$] the tuple $b$ is $\ideal_y$-wide over $M,a$, and the tuple $a^\prime$ is $\ideal_x$-wide over $M,b$,
			\item[$(3)$] the formula $\varphi(a,b)$ holds but $\varphi(a^\prime,b)$ does not.
		\end{enumerate}
		Using this, we will construct inductively a sequence $(a_i,b_i)_{i\geq 1}$ such that each finite sequence $(a_1,b_1,\ldots,a_k,b_k)$ is $M$-independent, and additionally $a_i\equiv_Ma$, $b_i\equiv_Mb$ and $\varphi(a_i,b_j)$ holds if and only if $i\leq j$. Clearly, this construction contradicts the $\ideal$-stability of $\varphi(x,y)$.
		
		Define $a_1=a$ and $b_1=b$. If $(a_1,b_1,\ldots,a_k,b_k)$ has been constructed, consider the partial type $\Phi(x,y)=p(x)\cup\{\neg\varphi(x,y)\}$ over $M$, where $p(x)=\tpp(a/M)$. By assumption we have that $\Phi(a^\prime,b)$ holds. Then, since $a^\prime$ is wide over $M,b$, the partial type $\Phi(x,b)$ is $\ideal_x$-wide. On the other hand, the tuple $b$ is $\ideal_y$-wide over $M$, and $(b_1,\ldots,b_k)$ is an $M$-independent tuple formed by realizations of $\tpp(b/M)$. Hence, by Lemma \ref{wideinter}, the partial type $\Phi(x,b_1)\cup\ldots\cup\Phi(x,b_k)$ is $\ideal_x$-wide. Set $a_{k+1}$ to be a realization of this partial type that is $\ideal_x$-wide over $M,a_{\leq k},b_{\leq k}$. Similarly, let $\Psi(x,y)=q(y)\cup\{\varphi(x,y)\}$, where $q(y)=\tpp(b/M)$. Now we have that $(a,b)$ is an $M$-independent tuple that realizes the partial type $\Psi(x,y)$, so $\Psi(a,y)$ is $\ideal_y$-wide. As before, the tuple $(a_1,\ldots,a_{k+1})$ is $M$-independent and it is formed by realizations of $\tpp(a/M)$, so $\Psi(a_1,y)\cup\ldots\cup\Psi(a_{k+1},y)$ is $\ideal_y$-wide by Lemma \ref{wideinter}. Let $b_{k+1}$ an element realizing this partial type that is $\ideal_y$-wide over $M,a_{\leq k+1},b_{\leq k}$. By induction, we obtain the required sequences.
		
		Finally, suppose that $(a^\prime,b^\prime)$ is $M$-independent, and let $a^{\prime\prime}\equiv_Ma$ be $\ideal_x$-wide over $M,b$. Then $(b,a^{\prime\prime})$ is $M$-independent and by the first part of the proof we obtain that $\varphi(a^{\prime\prime},b)$ holds. Equivalently, the formula $\varphi^*(b,a^{\prime\prime})$ holds. Again by the first part of the proof applied to the formula $\varphi^*(y,x)$ and the tuples $(a^{\prime\prime},b)$ and $(a^\prime,b^\prime)$, the formula $\varphi^*(b^\prime,a^\prime)$ also holds, that is  $\varphi(a^\prime,b^\prime)$ holds, as we wanted to prove.
	\end{proof}
	
	\begin{cor}\label{permutation}
		Suppose that $\varphi(x,y)$ is $\ideal$-stable of ladder $k$. Then there is no tuple $(a_1,b_1,\ldots,a_k,b_k)$ and no model $M$ such that $\varphi(a_i,b_j)$ holds if and only if $i\leq j$ and such that some permutation in $S_{2k}$ of $(a_1,b_1,\ldots,a_k,b_k)$ is $M$-independent.
	\end{cor}
	
	\begin{proof}
		Let $(a_1,b_1,\ldots,a_k,b_k)$ be a tuple such that $\varphi(a_i,b_j)$ holds if and only if $i\leq j$, and suppose to obtain a contradiction that there is some permutation in $S_{2k}$ of $(a_1,b_1,\ldots,a_k,b_k)$ which is independent over some model $M$. We will construct inductively an $M$-independent tuple $(c_1,d_1,\ldots,c_k,d_k)$ contradicting the $\ideal$-stability of ladder $k$ of $\varphi(x,y)$.
		
		Start with $c_1=a_1$, which is $\ideal_x$-wide over $M$. Assume that we have already found $c_1,d_1,\ldots,c_n$ for some $n<k$ as required. Since $b_n$ is $\ideal_y$-wide over $M$, there exists $d_n\equiv_Mb_n$ that is $\ideal_y$-wide over $M,c_{\leq n},d_{\leq n-1}$. Similarly, $a_{n+1}$ is $\ideal_x$-wide over $M$, so there is some $c_{n+1}\equiv_Ma_{n+1}$ that is $\ideal_x$-wide over $M,c_{\leq n},d_{\leq n}$. Hence, by induction, we obtain a sequence $(c_1,d_1,\ldots,c_k,d_k)$ which is $M$-independent. Therefore, we will be done if we prove that $\varphi(c_i,d_j)$ holds if and only if $i\leq j$. For every $1\leq i,j\leq k$ we have that:
		\begin{itemize}
			\item $(a_i,b_j)$ or $(b_j,a_i)$ is $M$-independent (by hypothesis),
			\item $(c_i,d_j)$ or $(d_j,c_i)$ is $M$-independent (by construction),
			\item $a_i\equiv_Mc_i$ and $b_j\equiv_Md_j$.
		\end{itemize}
		Then, by Theorem \ref{stat}, the formula $\varphi(c_i,d_j)$ holds if and only $i\leq j$.
	\end{proof}
	
	Recall that a $\varphi$-formula is a Boolean combination of formulas of the form $\varphi(x,b)$ and $\neg\varphi(x,b)$. A global $\varphi$-type is a maximal consistent set of $\varphi$-formulas. The aim of the rest of this section is to obtain some definability results for wide global $\varphi$-types when $\varphi(x,y)$ is an $\ideal$-stable formula.
	
	\begin{lemma}[Invariance]\label{inv}
		Let $\varphi(x,y)$ be an $\ideal$-stable formula and let $\gl{p}(x)$ be an $\ideal_x$-wide global $\varphi$-type. If $b$ and $b^\prime$ are $\ideal_y$-wide over a model $M$ and $b\equiv_Mb^\prime$, then $\varphi(x,b)\in\gl{p}$ if and only if $\varphi(x,b^\prime)\in\gl{p}$.
	\end{lemma}
	
	\begin{proof}
		Let $b$ and $b^\prime$ be $\ideal_y$-wide tuples over $M$ such that $b\equiv_Mb^\prime$, and assume for a contradiction that $\varphi(x,b)\land\neg\varphi(x,b^\prime)\in\gl{p}$. Let $c$ be a realization of $\gl{p}\vert_{M,b,b^\prime}$ that is $\ideal_y$-wide over $M,b,b^\prime$. Then both $(b,c)$ and $(b^\prime,c)$ are $M$-independent, but $\varphi(b,c)$ holds and $\varphi(b^\prime,c)$ does not, contradicting Theorem \ref{stat}.
	\end{proof}
	
	\begin{lemma}\label{def1}
		Let $\varphi(x,y)$ be an $\ideal$-stable formula. Let $\gl{p}(x)$ be an $\ideal_x$-wide global $\varphi$-type, and let $M$ be a model. Then for every $c\in\MB^{x}$ such that $\tpp(c/M)\cup\gl{p}(x)$ is $\ideal_x$-wide and every $b$ that is $\ideal_y$-wide over $M,c$, we have:
		\begin{align*}
			\varphi(x,b)\in\gl{p}\Leftrightarrow\varphi(c,b)\emph{ holds}.
		\end{align*}
	\end{lemma}
	
	\begin{proof}
		Let $c\in\MB^{x}$ be as in the statement. Observe that in particular $c$ is $\ideal_x$-wide over $M$. Now, for every $b$ that is $\ideal_y$-wide over $M,c$, consider a realization $c^\prime$ of $\tpp(c/M)\cup\gl{p}(x)\vert_{M,b}$ that is $\ideal_x$-wide over $M,b$. Then both $(c,b)$ and $(b,c^\prime)$ are $M$-independent tuples, and $c\equiv_Mc^\prime$. Hence, by the choice of $c^\prime$ and Theorem \ref{stat}, we have that
		\begin{align*}
			\varphi(x,b)\in\gl{p}\Leftrightarrow\varphi(c^\prime,b)\textrm{ holds}\Leftrightarrow\varphi(c,b)\textrm{ holds},
		\end{align*}
		as desired.
	\end{proof}
	
	The next corollary is an immediate consequence of this definability of types.
	
	\begin{cor}[Harrington for wide types]\label{harr}
		Let $\varphi(x,y)$ be an $\ideal$-stable formula and let $M$ be a model. If $\gl{p}(x)$ is an $\ideal_x$-wide global $\varphi$-type and $\gl{q}(y)$ is an $\ideal_y$-wide global $\varphi^*$-type, then 
		\begin{align*}
			\varphi(x,d)\in\gl{p}(x)\Leftrightarrow\varphi(c,y)\in\gl{q}(y),
		\end{align*}
		for every $c\in\MB^{x}$, $d\in\MB^{y}$ such that $\tpp(c/M)\cup\gl{p}(x)$ is $\ideal_x$-wide and $\tpp(d/M)\cup\gl{q}(y)$ is $\ideal_y$-wide.
	\end{cor}
	
	\begin{proof}
		Since $c$ is $\ideal_x$-wide over $M$, we can find $c^\prime\equiv_Mc$ that is $\ideal_x$-wide over $M,d$. Similarly, as $d$ is $\ideal_y$-wide over $M$, there exists some $d^\prime\equiv_Md$ that is $\ideal_y$-wide over $M,c$. By Lemma \ref{inv} and Lemma \ref{def1},
		\begin{align*}
			\varphi(x,d)\in\gl{p}(x)\Leftrightarrow\varphi(x,d^\prime)\in\gl{p}(x)\Leftrightarrow\varphi(c,d^\prime)\textrm{ holds}
		\end{align*}
		and
		\begin{align*}
			\varphi(c,y)\in\gl{q}(y)\Leftrightarrow\varphi(c^\prime,y)\in\gl{q}(y)\Leftrightarrow\varphi(c^\prime,d)\textrm{ holds}.
		\end{align*}
		The equivalence follows from Theorem \ref{stat}.
	\end{proof}
	
	We conclude this part by proving that if $\varphi(x,y)$ is $\ideal$-stable, then the set of tuples $b$ that are $\ideal_y$-wide over a model $M$ and for which $\varphi(x,b)$ belongs to a given $\ideal_x$-wide global $\varphi$-type is relatively definable in the set of $\ideal_y$-wide tuples over $M$.
	
	\begin{theorem}\label{def2}
		Let $\varphi(x,y)$ be an $\ideal$-stable formula. Let $\gl{p}(x)$ be an $\ideal_x$-wide global $\varphi$-type, and let $M$ be a model. Then there exists an $L(M)$-formula $\psi(y)$ such that for every $b$ that is $\ideal_y$-wide over $M$,
		\begin{align*}
			\varphi(x,b)\in\gl{p}\Leftrightarrow \psi(b)\emph{ holds}.
		\end{align*}
	\end{theorem}
	
	\begin{proof}
		We will prove that the set
		\begin{align*}
			X_{\gl{p},\varphi}=\{b\in\MB^{y}:b\textrm{ is $\ideal_y$-wide over $M$ and $\varphi(x,b)\in\gl{p}$}\}
		\end{align*}
		is relatively definable over $M$ in the set of $\ideal_y$-wide tuples over $M$. Let $p\in S(M)$ be a complete $\ideal_x$-wide type such that $p(x)\cup\gl{p}(x)$ is $\ideal_x$-wide.
		
		\begin{claim}
			Fix some realization $c_0$ of $p(x)$. Then, the set
			\begin{align*}
				Y=\{(c,b)\in\MB^{x}\times\MB^{y}:\emph{$c$ realizes $p(x)$ and $b$ is $\ideal_y$-wide over $M,c$}\}
			\end{align*}
			is type-definable over $M$ by the partial type
			\begin{align*}
				\Phi(x,y)=p(x)\cup\{\neg\phi(x,y)\in L(M):\phi(c_0,y)\in\ideal_y\}.
			\end{align*}
		\end{claim}
		
		\begin{claimproof}
			Suppose that $(c,b)\in Y$. Obviously, $c$ realizes $p(x)$. Now let $\phi(x,y)$ be an $L(M)$-formula such that $\phi(c_0,y)\in\ideal_y$. Since $c_0\equiv_Mc$ and $\ideal_y$ is $\emptyset$-invariant, then $\phi(c,y)\in\ideal_y$. This implies that $\neg\phi(c,b)$ holds, because $b$ is $\ideal_y$-wide over $M,c$. In conclusion, $(c,b)$ realizes $\Phi(x,y)$. Similarly, assume that $(c,b)$ realizes $\Phi(x,y)$. It is again obvious that $c$ realizes $p(x)$, and we only have to prove that $b$ is $\ideal_y$-wide over $M,c$. If not, there is a formula $\phi(x,y)\in L(M)$ such that $\phi(c,y)\in\ideal_y$ and $\phi(c,b)$ holds. As before, by the $\emptyset$-invariance of $\ideal_y$ we obtain that $\phi(c_0,y)\in \ideal_y$, which contradicts the fact that $(c,b)$ realizes $\Phi(x,y)$.
		\end{claimproof}
		
		\begin{claim}
			The set $X_{\gl{p},\varphi}$ is type definable over $M$ by the partial type
			\begin{align*}
				\Psi(y)=\exists x(\Phi(x,y)\cup\{\varphi(x,y)\}).
			\end{align*}
		\end{claim}
		
		\begin{claimproof}
			If $b$ realizes $\Psi(y)$, then there is some $c$ realizing $p(x)$ such that $b$ is $\ideal_y$-wide over $M,c$ and $\varphi(c,b)$ holds. By Lemma \ref{def1} $\varphi(x,b)\in\gl{p}$, that is, $b\in X_{\gl{p},\varphi}$. On the other hand, if $b\in X_{\gl{p},\varphi}$, then $\varphi(x,b)\in\gl{p}$ and $b$ is $\ideal_y$-wide over $M$. Now let $c^\prime$ be any realization of $p(x)$ and let $b^\prime\equiv_Mb$ be an $\ideal_y$-wide tuple over $M,c^\prime$. Then there exists an automorphism $f\in\Aut_M(\MB)$ such that $f(b^\prime)=b$. The tuple $c:=f(c^\prime)$ realizes $p(x)$, and $b$ is $\ideal_y$-wide over $M,c$. Hence, $\varphi(c,b)$ holds by Lemma \ref{def1} and $b$ realizes $\Psi(y)$.
		\end{claimproof}
		
		Consider the topological space $S_y(M)$, and let $W\subset S_y(M)$ be the closed subset formed by the $\ideal_y$-wide types. With the induced topology, $W$ is compact, Hausdorff and $0$-dimensional, and a basis of clopen sets is given by $[\psi]_W=\{p\in W:\psi\in p\}$ for any formula $\psi(y)\in L(M)$ with $\psi(y)\notin\ideal_y$.
		
		By the previous claim, $X_{\gl{p},\varphi}$ is type-definable over $M$ by a partial type $\Psi(y)$. Thus, the set $[X_{\gl{p},\varphi}]$ defined by
		\begin{align*}
			\bigcap_{\psi\in\Psi}[\psi]=\{\tpp(b/M):b\in X_{\gl{p},\varphi}\}
		\end{align*}
		is a closed subset of $S_y(M)$ that is contained in $W$, so it is a closed subset of $W$. Similarly, since $\neg\varphi(x,y)$ is $\ideal$-stable, the set $[X_{\gl{p},\neg\varphi}]$ is also a closed subset of $W$. But $[X_{\gl{p},\neg\varphi}]=W\setminus[X_{\gl{p},\varphi}]$, so $[X_{\gl{p},\varphi}]$ is clopen in $W$, and by compactness there is some $L(M)$-formula $\psi(y)$ such that $\psi(y)\notin\ideal_y$ and $[X_{\gl{p},\varphi}]=[\psi]_W$. Hence, $X_{\gl{p},\varphi}$ is relatively definable over $M$ by $\psi(y)$ in the set of $\ideal_y$-wide tuples over $M$.
	\end{proof}
	
	Given an $\ideal_x$-wide global $\varphi$-type $\gl{p}(x)$ and a model $M$, we will denote an $L(M)$-formula $\psi(y)$ as in the above theorem by $\deftp{\gl{p},M}{\varphi(y)}$, so for every tuple $b$ that is $\ideal_y$-wide over $M$ we have that
	\begin{align*}
		\varphi(x,b)\in\gl{p}\Leftrightarrow\deftp{\gl{p},M}{\varphi(b)}\textrm{ holds}.
	\end{align*}
	In contrast with the stable context, the formula $\deftp{\gl{p},M}{\varphi(y)}$ need not be a Boolean combination of instances of $\varphi^*(y,x)$.
	
	\section{Topological constructions}\label{topology}
	
	Recall that $\ideal_x$ and $\ideal_y$ are $\emptyset$-invariant $S_1$-ideals defined in, respectively, $L_x(\MB)$ and $L_y(\MB)$. The results from the previous section show that when $\varphi(x,y)$ is $\ideal$-stable, we have a very good understanding of the behavior of $\ideal_x$-wide global $\varphi$-types, modulo formulas with parameters that are not $\ideal_y$-wide over a fixed model. With this motivation, we introduce a new topology in the set of $\varphi$-types. The clopen sets of the usual topology in $S_\varphi(\MB)$ are the sets of the form
	\begin{align*}
		[\psi(x)]=\{\gl{p}\in S_\varphi(\MB):\psi\in\gl{p}\}
	\end{align*}
	for every $\varphi$-formula $\psi(x)$. We now modify this topology:

	\begin{definition}
		Let $M$ be a model and let $\varphi(x,y)$ be an arbitrary formula. Let $\wide{M}$ be the set of elements of $\MB^{y}$ that are $\ideal_y$-wide over $M$. We define $\tau_M$ to be the topology in $S_\varphi(\MB)$ generated by all the finite intersections of sets of the form $[\varphi(x,b)]$ and $[\neg\varphi(x,b)]$ when $b$ ranges over $\wide{M}$.
	\end{definition}
	
	Consider the set $S_\varphi^{\rm w}(\MB)$ formed by the $\ideal_x$-wide global $\varphi$-types. With the induced topology, the space $\left(S_\varphi^{\rm w}(\MB),\tau_M\right)$ is $0$-dimensional and compact, but not Hausdorff.
	
	\begin{lemma}\label{clequiv}
		For all $\gl{p}(x)$ and $\gl{q}(x)$ in $S_\varphi^{\rm w}(\MB)$ the following are equivalent:
		\begin{enumerate}
			\item[$(1)$] $\varphi(x,b)\in\gl{p}$ if and only if $\varphi(x,b)\in\gl{q}$ for every $b$ that is $\ideal_y$-wide over $M$;
			\item[$(2)$] $\gl{p}\in\rm{cl}(\gl{q})$, i.e. $\gl{q}$ is a specialization of $\gl{p}$;
			\item[$(3)$] $\rm{cl}(\gl{p})=\rm{cl}(\gl{q})$.
		\end{enumerate}
	\end{lemma}
	
	\begin{proof}
		To prove that $(1)$ implies $(2)$, observe that if $\gl{p}\notin\rm{cl}(\gl{q})$, then we can find a clopen set $U$ of the form $[\varphi(x,b)]$ or $[\neg\varphi(x,b)]$ with $b\in\wide{M}$ such that $\gl{p}\in U$ and $\gl{q}\notin U$. This clearly contradicts $(1)$. To see that $(2)$ implies $(3)$, it is enough to prove that $\gl{q}\in\rm{cl}(\gl{p})$. Otherwise, we can find a clopen set $U$ as above such that $\gl{q}\in U$ and $\gl{p}\notin U$. Then $S_\varphi^{\rm w}(\MB)\setminus U$ is an open set that contains $\gl{p}$ and that does not contain $\gl{q}$, which contradicts point $(2)$. Finally, it is obvious that $(3)$ implies $(1)$: if $\gl{p}$ and $\gl{q}$ differ in a formula of the form $\varphi(x,b)$ with $b\in\wide{M}$, then this formula yields the existence of a clopen set $U$ such that $\gl{p}\in U$ and $\gl{q}\notin U$, which is not possible by $(3)$.
	\end{proof}
	
	In light of the previous result, we define an equivalence relation on $S_\varphi^{\rm w}(\MB)$ as follows:
	\begin{align*}
		\gl{p}\sim_M\gl{q}\Leftrightarrow\rm{cl}(\gl{p})=\rm{cl}(\gl{q}).
	\end{align*}
	Then the space $\left(S_\varphi^{\rm w}(\MB),\tau_M\right)/\sim_M$ is compact, Hausdorff and $0$-dimensional. Moreover, if $\varphi(x,y)$ is $\ideal$-stable this construction does not depend on the model $M$:
	
	\begin{lemma}\label{samecl}
		Let $M\prec N$ be two models, and assume that $\varphi(x,y)$ is $\ideal$-stable. Then the topologies $\tau_M$ and $\tau_N$ coincide in $S_\varphi^{\rm w}(\MB)$.
	\end{lemma}
	
	\begin{proof}
		To avoid any confusion, denote by $\textrm{cl}_M$ and $\textrm{cl}_N$ the closure operators in, respectively, $\tau_M$ and $\tau_N$. It suffices to show that for all $\ideal_x$-wide global $\varphi$-types $\gl{p}$ and $\gl{q}$, we have that $\gl{p}\in\textrm{cl}_M(\gl{q})$ if and only if $\gl{p}\in\textrm{cl}_N(\gl{q})$. Assume first that $\gl{p}\in\textrm{cl}_N(\gl{q})$, and let $b$ be a tuple that is $\ideal_y$-wide over $M$. Then we can find a tuple $b^\prime\equiv_Mb$ that is $\ideal_y$-wide over $N$, so by Lemma \ref{inv} and Lemma \ref{clequiv} we have that
		\begin{align*}
			\varphi(x,b)\in\gl{p}\Leftrightarrow\varphi(x,b^\prime)\in\gl{p}\Leftrightarrow\varphi(x,b^\prime)\in\gl{q}\Leftrightarrow\varphi(x,b)\in\gl{q}.
		\end{align*}
		Hence, $\gl{p}\in\rm{cl}_M(\gl{q})$. The other direction is trivial, since $\wide{N}\subset\wide{M}$.
	\end{proof}
	
	We now define a new topological space which we will use to identify the elements of the quotient $\left(S_\varphi^{\rm w}(\MB),\tau_M\right)/\sim_M$ with certain partial types.
	
	\begin{definition}\label{XM}
		Let $M$ be a model and let $\varphi(x,y)$ be an arbitrary formula. Let $\gl{X}_{\varphi,M}$ be the set of $\ideal_x$-wide partial $\varphi$-types that are maximal consistent sets of Boolean combinations $\psi(x)$ of formulas of the form $\varphi(x,b)$ and $\neg\varphi(x,b)$ when $b$ ranges over $\wide{M}$. When the context is clear we refer to this set just as $\gl{X}_M$.
	\end{definition}
	
	The set $\gl{X}_M$ has a natural structure of $0$-dimensional topological space, where a basis of clopen sets is given by
	\begin{align*}
		[\psi(x)]_M=\{\gl{p}\in\gl{X}_M:\psi(x)\in\gl{p}\},
	\end{align*}
	for every formula $\psi(x)$ as in Definition \ref{XM}. With this topology, the space $\gl{X}_M$ is Hausdorff. Moreover, since the restriction map $S_\varphi^{\rm w}(\MB)\to\gl{X}_M$ defined as $\gl{p}\mapsto\gl{p}\vert_{\wide{M}}$ is continuous, then $\gl{X}_M$ is compact. Note also that for every two models $M\prec N$, the map
	\begin{align*}
		\rho_{M,N}:\gl{X}_M\to\gl{X}_N\textrm{, }\gl{p}\mapsto\gl{p}\vert_{\wide{N}}
	\end{align*}
	is continuous and surjective.

	\begin{remark}\label{spaceshomeo}
		The restriction $S_\varphi^{\rm{w}}(\MB)\to\gl{X}_M$ defined above is surjective. Hence, the continuous bijection that it induces from the compact space $\left(S_\varphi^{\rm w}(\MB),\tau_M\right)/\sim_M$ into the Hausdorff space $\gl{X}_M$ is indeed an homeomorphism. In particular, when $\varphi(x,y)$ is $\ideal$-stable, for every two models $M\prec N$ we have by Lemma \ref{samecl} the following commutative diagram:
		\[
		\begin{tikzcd}[row sep=huge, column sep=huge]
			\left(S_\varphi^{\rm w}(\mathbb{M}),\tau_M\right)/\sim_M 
			\arrow[d, "\cong"'] 
			\arrow[r, "\mathrm{id}", "\cong"']
			& \left(S_\varphi^{\rm w}(\mathbb{M}),\tau_N\right)/\sim_N 
			\arrow[d, "\cong"] \\
			\mathfrak{X}_M 
			\arrow[r, "\rho_{M,N}"] 
			& \mathfrak{X}_N.
		\end{tikzcd}
		\]
		It follows that $\rho_{M,N}$ is an homeomorphism.
	\end{remark}
	
	\begin{remark}\label{resultsforXM}
		Assume that $\varphi(x,y)$ is $\ideal$-stable, and fix $\gl{p}\in \gl{X}_M$. On the one hand, it is clear by Lemma \ref{inv} that $\gl{p}$ is $M$-invariant, that is if $b$ and $b^\prime$ are $\ideal_y$-wide over $M$ and $b\equiv_Mb^\prime$, then $\varphi(x,b)\in\gl{p}$ if and only if $\varphi(x,b^\prime)\in\gl{p}$. On the other hand, if $c\in\MB^{x}$ is such that $\tpp(c/M)\cup\gl{p}(x)$ is $\ideal_x$-wide, then $\gl{p}$ can be extended to an $\ideal_x$-wide global $\varphi$-type $\gl{q}$ such that $\tpp(c/M)\cup\gl{q}(x)$ is $\ideal_x$-wide. Hence, by Lemma \ref{def1}, for every $b$ that is $\ideal_y$-wide over $M,c$ we have that
		\begin{align*}
			\varphi(x,b)\in\gl{p}\Leftrightarrow\varphi(x,b)\in\gl{q}\Leftrightarrow\varphi(c,b)\textrm{ holds}.
		\end{align*}
		Similarly, using the notation introduced after Theorem \ref{def2}, if $\gl{q}$ is any wide global $\varphi$-type such that $\gl{q}\vert_{\wide{M}}=\gl{p}$ then, for any $b$ that is $\ideal_y$-wide over $M$,
		\begin{align*}
			\varphi(x,b)\in\gl{p}\Leftrightarrow\varphi(x,b)\in\gl{q}\Leftrightarrow\deftp{\gl{q},M}{\varphi(b)}\textrm{ holds}.
		\end{align*}
	\end{remark}

	A classical result in stability theory states that if $\varphi(x,y)$ is stable, then the space of $\varphi$-types $S_\varphi(\MB)$ has finite Cantor-Bendixson rank \cite[Lemma 3.1]{AP96}. We now prove that the same is true for the space $\gl{X}_M$ in the $\ideal$-stable case. We will use the notions of sequence and binary tree (Definition \ref{bintree}) and a well-known result by Hodges (Theorem \ref{tree2}).
	
	\begin{theorem}\label{finiteCB}
		Let $M$ be a model and suppose that $\varphi(x,y)$ is $\ideal$-stable of ladder $k$. Then the Cantor-Bendixson rank of $\gl{X}_M$ is strictly less than $2^{k+1}-2$.
	\end{theorem}
	
	\begin{proof}
		Consider the linear order $\prec$ defined on $^{<\omega}2$ as follows. Let $\sigma\in{^r2}$ and $\tau\in{^s2}$. If $r\neq s$, then $\sigma\prec\tau$ if and only if $r<s$, and if $r=s$, then $\sigma\prec\tau$ if and only if $\sigma$ is smaller than $\tau$ with respect to the lexicographic order on $^r2$. Denote by $\gl{X}_M^{(n)}$ the $n$-th Cantor-Bendixson derivative of $\gl{X}_M$.
		
		\begin{claim*}
			If $\gl{X}_M^{(n)}\neq\emptyset$, then we can find $(\gl{p}_\sigma:\sigma\in{^n2})$ in $\gl{X}_M$ and tuples $(b_\tau:\tau\in{^{<n}2})$ such that:
			\begin{enumerate}
				\item[$(1)$] $(b_\tau:\tau\in{^{<n}2})$ is $M$-independent with respect to the order $\prec$,
				\item[$(2)$] $\varphi(x,b_{\sigma\vert i})^{\sigma(i)}\in\gl{p}_\sigma$ for every $i<n$ and every $\sigma\in{^n2}$, where $\varphi^0=\varphi$ and $\varphi^1=\neg\varphi$. 
			\end{enumerate}
		\end{claim*}
		
		\begin{claimproof*}
			Let $\gl{p}_0\in\gl{X}_M^{(n)}$. Then $\gl{p}_0\in\gl{X}_M^{(n-1)}$ and $\gl{p}_0$ is not isolated in $\gl{X}_M^{(n-1)}$ by the clopen set $\gl{X}_M$, so there is $\gl{p}_1\in\gl{X}_M^{(n-1)}$ such that $\gl{p}_0\neq\gl{p}_1$. Up to relabelling $\gl{p}_0$ and $\gl{p}_1$, we may assume that there exists a tuple $b_\emptyset$ such that $\varphi(x,b_\emptyset)\in\gl{p}_0$ and $\neg\varphi(x,b_\emptyset)\in\gl{p}_1$. Observe that $b_\emptyset$ is $\ideal_y$-wide over $M$. Since $\gl{p}_0$ and $\gl{p}_1$ are in $\gl{X}_M^{(n-1)}$, this process can be iterated. More precisely, assume that we have already defined $b_{\prec\tau}$ for some $\tau\in{^d2}$ and $d<n$, and that there exists $\gl{p}_\tau\in\gl{X}_M^{(n-d)}$ such that
			\begin{align*}
				\varphi(x,b_{\tau\vert i})^{\tau(i)}\in\gl{p}_\tau\textrm{ for every }i<d.
			\end{align*}
			Define $\gl{p}_{\tau^\smallfrown 0}=\gl{p}_\tau$. As before, $\gl{p}_{\tau^\smallfrown 0}\in\gl{X}_M^{(n-d-1)}$ and $\gl{p}_{\tau^\smallfrown 0}$ is not isolated in $\gl{X}_M^{(n-d-1)}$ by the clopen set
			\begin{align*}
				\left[\varphi(x,b_\emptyset)^{\tau(0)}\land\ldots\land\varphi(x,b_{\tau\vert d-1})^{\tau(d-1)}\right]_M,
			\end{align*}
			so we can find $\gl{p}_{\tau^\smallfrown 1}\in\gl{X}_M^{(n-d-1)}$ with $\gl{p}_{\tau^\smallfrown 0}\neq\gl{p}_{\tau^\smallfrown 1}$ and
			\begin{align*}
				\varphi(x,b_{\tau\vert i})^{\tau(i)}\in\gl{p}_{\tau^\smallfrown 1}\textrm{ for every }i<d.
			\end{align*}
			Up to relabelling $\gl{p}_{\tau^\smallfrown 0}$ and $\gl{p}_{\tau^\smallfrown 1}$, we may assume that there exists a tuple $b_\tau$ such that $\varphi(x,b_\tau)^0\in \gl{p}_{\tau^\smallfrown 0}$ and $\varphi(x,b_\tau)^1\in\gl{p}_{\tau^\smallfrown 1}$. Since $b_\tau$ is $\ideal_y$-wide over $M$, we may also assume by Lemma \ref{inv} that $b_\tau$ is $\ideal_y$-wide over $M,b_{\prec\tau}$. Therefore, by induction our claim holds.
		\end{claimproof*}
		
		Finally, let $n=2^{k+1}-2$. If the Cantor-Bendixson rank of $\gl{X}_M$ is at least $n$, then in particular $\gl{X}_M^{(n)}\neq\emptyset$, and we can find $(\gl{p}_\sigma:\sigma\in{^n2})$ in $\gl{X}_M$ and $(b_\tau:\tau\in{^{<n}2})$ as in the claim. Now, for every $\sigma\in{^n2}$, set recursively on $\prec$ the tuple $a_\sigma$ to be a realization of $\gl{p}_\sigma\vert_{M,a_{\prec\sigma},b_{\prec\sigma}}$ that is $\ideal_x$-wide over $M,a_{\prec\sigma},b_{\prec\sigma}$. Then we have obtained a binary tree for $\varphi(x,y)$ of height $n$ with the additional property that the tuple $(b_\tau:\tau\in{^{<n}2})^{\smallfrown}(a_\sigma:\sigma\in{^n2})$ is $M$-independent with respect to $\prec$. By Theorem \ref{tree2}, there are $a_0,\ldots,a_{k-1}\in\{a_\sigma:\sigma\in{^n2}\}$ and $b_0,\ldots,b_{k-1}\in\{b_\tau:\tau\in{^{<n}2}\}$ such that $\varphi(a_i,b_j)$ holds if and only if $i\leq j$. Since some permutation of $(a_0,b_0,\ldots,a_{k-1},b_{k-1})$ is $M$-independent with respect to the usual order by construction, we obtain a contradiction by Corollary \ref{permutation}. Thus, the Cantor-Bendixson rank of $\gl{X}_M$ is strictly less than $2^{k+1}-2$.
	\end{proof}
	
	\section{A graph regularity lemma for almost stable relations}\label{regularity}
	
	In this section we will not work with arbitrary $S_1$-ideals, but rather with ideals given by Keisler measures. Therefore, let $\mu=\mu_x$ and $\nu=\nu_y$ be two $\emptyset$-invariant global Keisler measures, and let $\ideal_x$ and $\ideal_y$ be the ideals formed, respectively, by the sets of $\mu$-measure zero and the sets of $\nu$-measure zero. Recall by Example \ref{measures} that in this case $\ideal_x$ and $\ideal_y$ are $\emptyset$-invariant and $S_1$. Instead of writing $\ideal_x$-wide and $\ideal_y$-wide, we will write $\mu$-wide and $\nu$-wide.
	
	\begin{definition}
		We say that a formula $\varphi(x,y)$ is \emph{almost sure stable with respect to $\mu$ and $\nu$} if it is $(\ideal_x,\ideal_y)$-stable according to Definition \ref{Istabformula}, where $\ideal_x$ and $\ideal_y$ are the ideals given by $\mu$ and $\nu$ described in the paragraph above.
	\end{definition}
	
	As usual, we identify every Keisler measure $\mu$ with its unique extension to a Borel regular probability measure on $S_x(\MB)$ \cite[Section 7.1]{PS14}. Given any model $N$, we denote by $\mu\vert_N$ the Borel regular probability measure on $S_x(N)$ associated with the restriction of $\mu$ to $L_x(N)$. Observe that if $r:S_x(\MB)\to S_x(N)$ is the canonical restriction map, then $\mu\vert_N$ is the pushforward measure $r_*(\mu)$, since for every clopen subset $U$ of $S_x(N)$ we have that
	\begin{align*}
		\mu\vert_N(U)=\mu(r^{-1}(U)),
	\end{align*}
	and by regularity the same holds for every Borel subset of $S_x(N)$.
	
	Indeed, the proof of \cite[Section 7.1]{PS14} shows that the above extension of finitely additive probability measures to Borel regular probability measures is a general topological fact:
	\begin{fact}\label{extension}
		If $X$ is a compact, Hausdorff and $0$-dimensional topological space and $\mu$ is a finitely additive probability measure on the clopen sets of $X$, then $\mu$ extends uniquely to a Borel regular probability measure on $X$.
	\end{fact}
	We now apply this result to the topological constructions from the previous section. For every formula $\varphi(x,y)\in L$ and every model $M$, the measure $\mu$ induces a finitely additive probability measure on the clopen sets of the space $\gl{X}_{\varphi,M}$. Hence, the measure $\mu$ (more precisely, its restriction to the Boolean algebra of the formulas associated with the clopen sets of $\gl{X}_{\varphi,M}$) can be extended to a unique regular probability measure on $\gl{X}_{\varphi,M}$. Denote by $\mu_{\varphi,M}$ such extension.
	
	\begin{remark}\label{mscompatible}
		The measure $\mu_{\varphi,M}$ is the pushforward measure of $\mu$ via the continuous restriction map $r_{\varphi,M}:S_x(\MB)\to\gl{X}_{\varphi,M}$ given by $r_{\varphi,M}(\gl{p})=(\gl{p}\vert_\varphi)\vert_{\wide{M}}$, since for every clopen subset $U$ of $\gl{X}_{\varphi,M}$ we have that
		\begin{align*}
			\mu_{\varphi,M}(U)=\mu(r_{\varphi,M}^{-1}(U)),
		\end{align*}
		and the same holds by regularity for every Borel subset of $\gl{X}_{\varphi,M}$.
	\end{remark}
	
	To prove the main result of this section, namely Theorem \ref{regularmeasure}, we will make use of a result by Malliaris and Pillay concerning stable relations \cite[Lemma 2.1]{PM15} (see also \cite{AP20} and \cite[Theorem 5.3.9]{BB83}). This result can be formulated in purely topological terms as follows:
	
	\begin{fact}\label{partitions}
		Let $X$ be a compact, Hausdorff and $0$-dimensional topological space of finite Cantor-Bendixson rank, and let $\mu$ be a Borel regular probability measure on $X$. If $U\subset X$ is a clopen set with $\mu(U)>0$, then for every $\eps>0$ we can partition $U$ into finitely many pairwise disjoint clopen sets $U_1,\ldots,U_n$ in $U$ such that each $U_i$ contains a point $x_i$ with $\mu(\{x_i\})>0$ and
		\begin{align*}
			\mu(U_i)\leq\mu(\{x_i\})\cdot(1+\eps).
		\end{align*}
		In particular, we have $\mu(U_i\setminus\{x_i\})\leq\eps\cdot\mu(\{x_i\})$ for every $i=1,\ldots,n$.
	\end{fact}
	
	\begin{proof}
		The proof is merely an adaptation of \cite[Lemma 2.1]{PM15}, and is by induction on the Cantor-Bendixson rank of the clopen sets.
		
		Fix $\eps>0$ and assume that the result is true for every clopen subset of $X$ of positive measure and rank less than $m$. Let $U$ be a clopen set with $\mu(U)>0$ and rank $m$. By compactness, there are finitely many points of maximal rank, say $x_1,\ldots,x_k$, and write $\alpha_i=\mu(\{x_i\})$. Let $J=\{i:\alpha_i>0\}$.
		
		Suppose first that $J\neq\emptyset$, and fix $i_0\in J$. By regularity, we can find a clopen set $U_{i_0}$ in $U$ such that $U_{i_0}\cap U^{(m)}=\{x_{i_0}\}\cup\{x_i:i\notin J\}$ and with $\mu(U_{i_0})\leq\alpha_{i_0}\cdot(1+\eps)$. On the other hand, for every $i\in J\setminus\{i_0\}$, let $U_i$ be a clopen set in $U$ such that $U_i\cap U^{(m)}=\{x_i\}$ with $\mu(U_i)\leq\alpha_i\cdot(1+\eps)$. We may assume that the sets $U_i$ are all pairwise disjoint. Then the complement in $U$ of $\bigcup_{i\in J}U_i$ has either measure zero, in which case we can adjoin it to $U_{i_0}$ obtaining a suitable partition for $U$, or it is a clopen set of positive measure and rank strictly less than $m$, in which case we obtain by induction a partition for $U\setminus \bigcup_{i\in J}U_i$ as in the statement. This partition, together with $\bigcup_{i\in J}U_i$ yields the required partition for $U$.
		
		Finally suppose that $J=\emptyset$, and consider the set $U\setminus\{x_1,\ldots,x_k\}$, which is open and has positive measure. By regularity, we can find a clopen subset $U^\prime$ of $X$ contained in $U\setminus\{x_1,\ldots,x_k\}$ with $\mu(U^\prime)>0$. Since the Cantor-Bendixson rank of $U^\prime$ is less than $m$, then by induction $U^\prime$ contains a point $x$ of positive measure. Clearly, the Cantor-Bendixson rank of $x$ is less than $m$, so there exists some $n<m$ and a clopen set $W$ in  $U$ such that $x$ is isolated in $U^{(n)}$ by $W$. Now observe that $W$ has positive measure, because $x\in W$, and additionally $x_1,\ldots,x_k\notin W$. Therefore, the Cantor-Bendixson rank of $W$ is strictly less than $m$, and by induction we can partition $W$ into clopen sets $W_1,\ldots,W_r$ in $W$ such that each $W_i$ contains a point $y_i$ of positive measure and $\mu(W_i)\leq\mu(\{y_i\})\cdot(1+\eps)$. Moreover, we can fix $i_0\in\{1,\ldots,r\}$ and we can assume that $x_1,\ldots,x_k\in W_{i_0}$ just by adjoining to $W_{i_0}$ a clopen set $W^\prime$ in $U$ which is disjoint from the sets $W_i$, containing $x_1,\ldots,x_k$ and whose measure is $\mu(W^\prime)\leq \mu(\{y_{i_0}\})\cdot(1+\eps)-\mu(W_{i_0})$. We conclude as above by considering the complement in $U$ of $\bigcup_{i=1}^r W_i$.
	\end{proof}
	
	We will now construct a Keisler measure in $L_{xy}(\MB)$ which naturally extends both $\mu$ and $\nu$.
	\begin{definition}\cite[Definition 2.21]{KG20}
		A Keisler measure $\mu$ is \emph{Borel-definable over} $A$ if it is $A$-invariant and for every formula $\psi(x,y)\in L$ the map
		\begin{align*}
			F_{\mu,A}^\psi:S_y(A)\to[0,1]\textrm{, }\tpp(b/A)\mapsto\mu(\psi(x,b))
		\end{align*}
		is Borel-measurable, i.e. for every open set $U\subset[0,1]$, the set $(F_{\mu,A}^\psi)^{-1}(U)$ is Borel.
	\end{definition}
	
	\begin{remark}
		If $\mu$ is Borel-definable over $A$, then it is Borel-definable over any model $N$ containing $A$, and in that case, for every formula $\psi(x,y)\in L(N)$ we have that $F_{\mu,N}^\psi$ is a Borel-measurable map \cite[Proposition 2.22]{KG20}.
	\end{remark}
	
	Assume that the measure $\mu$ is Borel-definable over $\emptyset$. In particular, for every model $N$ and every $L(N)$-formula $\psi(x,y)$, the function $F_{\mu,N}^\psi$ is $\nu\vert_N$-measurable, and we can define the global measure $\mu\otimes\nu$ of definable sets in the variables $x$ and $y$ as follows:
	\begin{align*}
		\mu\otimes\nu(\psi(x,y))=\int_{q\in S_y(N)}F_{\mu,N}^\psi(q)\mathrm{d}\nu_y\vert_N=\int_{q\in S_y(N)}\mu(\psi(x,b))\mathrm{d}\nu_y\vert_N,
	\end{align*}
	where $b$ is some (any) realization of $q(y)$.
	
	\begin{remark}
		Observe that:
		\begin{enumerate}
			\item the definition of $\mu\otimes\nu$ does not depend on the choice of $N$ \cite[Section 7.4]{PS14};
			\item the above construction is not commutative in general, i.e. $\mu\otimes\nu\neq\nu\otimes\mu$;
			\item the measure $\mu\otimes\nu$ extends the product measure $\mu\times\nu$.
		\end{enumerate}
	\end{remark}
	
	\begin{theorem}\label{regularmeasure}
		Let $\mu=\mu_x$ and $\nu=\nu_y$ be two $\emptyset$-invariant Keisler measures, and assume that $\mu$ is Borel-definable over $\emptyset$. If $\varphi(x,y)$ is almost sure stable with respect to $\mu$ and $\nu$, then for every model $M$ and every $\eps>0$ there are formulas $\phi_1(x),\ldots,\phi_n(x)$ and $\theta_1(y),\ldots,\theta_m(y)$ such that
		\begin{enumerate*}
			\item[$(1)$] each $\phi_i$ is a Boolean combination of formulas of the form $\varphi(x,b)$ and $\neg\varphi(x,b)$ where the tuples $b$ are $\nu$-wide over $M$,
			\item[$(2)$] each $\theta_j$ is a Boolean combination of formulas of the form $\varphi(a,y)$ and $\neg\varphi(a,y)$ where the tuples $a$ are $\mu$-wide over $M$,
			\item[$(3)$] $\mu(\phi_1\lor\ldots\lor\phi_n)=1$ and $\nu(\theta_1\lor\ldots\lor\theta_m)=1$,
		\end{enumerate*}
		and such that for every $1\leq i\leq n$ and $1\leq j\leq m$, either
		\begin{align*}
			\mu\otimes\nu(\varphi(x,y)\land \phi_i(x)\land\theta_j(y))\geq(1-\eps)^2\cdot \mu\otimes\nu(\phi_i(x)\land\theta_j(y)) 
		\end{align*}
		or
		\begin{align*}
			\mu\otimes\nu(\neg\varphi(x,y)\land \phi_i(x)\land\theta_j(y))\geq(1-\eps)^2\cdot \mu\otimes\nu(\phi_i(x)\land\theta_j(y)).
		\end{align*}
	\end{theorem}
	
	\begin{proof}
		Let $\lambda=\mu\otimes\nu$ and let $\mu_{\varphi,M}$ and $\nu_{\varphi^*,M}$ be the extensions of $\mu$ and $\nu$ to Borel regular probability measures in the spaces $\gl{X}_{\varphi,M}$ and $\gl{X}_{\varphi^*,M}$. Given $\eps>0$, consider the partitions
		\begin{align*}
			\gl{X}_{\varphi,M}&=U_1\cup\ldots\cup U_n,\\
			\gl{X}_{\varphi^*,M}&=V_1\cup\ldots\cup V_m,
		\end{align*}
		into clopen sets given by Fact \ref{partitions}. The formulas corresponding to the clopen sets of the above partitions clearly satisfy $(1)$, $(2)$ and $(3)$. Now fix $i$ and $j$, and let $U_i=[\phi_i(x)]_M$ and $V_j=[\theta_j(y)]_M$, together with the associated types $\gl{p}_i\in U_i$ and $\gl{q}_j\in V_j$ given by the same fact. Let $p(x)\in S_x(M)$ be any complete $\mu$-wide type such that $p(x)\cup \gl{p}_i(x)$ is $\mu$-wide, and let $c$ be a realization of $p(x)$. Finally, let $N$ be an $\abs{M}^+$-saturated model containing $M,c$ and the parameters of $\phi_i$ and $\theta_j$. Henceforth, we work in the space $S_y(N)$.
		
		Recall that the measure $\mu_{\varphi,M}$ is the pushforward measure of $\mu$ via the map $r_{\varphi,M}:S_x(\MB)\to\gl{X}_{\varphi,M}$ defined in Remark \ref{mscompatible}. Similarly, the measure $\nu_{\varphi^*,M}$ is the pushforward of $\nu$ via the corresponding restriction map $r_{\varphi^*,M}:S_y(\MB)\to\gl{X}_{\varphi^*,M}$.
		
		Suppose first that $\varphi(c,y)\in \gl{q}_j$, so that $\theta_j(y)\land\varphi(c,y)\in \gl{q}_j$, and define the following closed subset of $S_y(N)$,
		\begin{align*}
			F=\{q\in S_y(N):q(y)\vert_{M,c}\textrm{ is $\nu$-wide}\}.
		\end{align*}
		If $\tpp(b/N)\in[\varphi(c,y)]\cap F$ then $\varphi(x,b)\in\gl{p}_i$ by Lemma \ref{def1}, and so we have that  $\varphi(x,b)\land\phi_i(x)\in \gl{p}_i$. Consequently, $\mu(\varphi(x,b)\land\phi_i(x))\geq\mu_{\varphi,M}(\gl{p}_i)$. Now since $\nu\vert_N(F)=1$, we obtain that $\nu\vert_N([\theta_j(y)\land\varphi(c,y)]\cap F)\geq\nu_{\varphi^*,M}(\gl{q}_j)$, and it follows that
		\begin{align*}
			\lambda(\varphi(x,y)\land\phi_i(x)\land\theta_j(y))&=\int_{\tpp(b/N)\in[\theta_j(y)]}\mu(\varphi(x,b)\land\phi_i(x))\mathrm{d}\nu_y\vert_N\\
			&\geq\int_{\tpp(b/N)\in [\theta_j(y)\land\varphi(c,y)]\cap F}\mu(\varphi(x,b)\land\phi_i(x))\mathrm{d}\nu_y\vert_N\\
			&\geq\mu_{\varphi,M}(\gl{p}_i)\cdot\nu\vert_N([\theta_j(y)\land\varphi(c,y)]\cap F)\\
			&\geq\mu_{\varphi,M}(\gl{p}_i)\cdot\nu_{\varphi^*,M}(\gl{q}_j)\\
			&\geq(1-\eps)^2\cdot\lambda(\phi_i(x)\land\theta_j(y)),
		\end{align*}
		as required. This finishes the first case. If $\varphi(c,y)\notin q_j(y)$, then $\neg\varphi(c,y)\in q_j(y)$, and proceeding similarly we obtain the second case. Hence, the result.
	\end{proof}
	
	As a corollary, we answer a question of \cite[Page 14]{AMP242} by giving a model-theoretic account of \cite[Theorem 6.13]{TW21} (see also \cite{EFR86}).
	
	\begin{cor}
		Under the assumptions of Theorem \ref{regularmeasure}, for every $\eps>0$ there exists a stable formula $\psi(x,y)$ such that
		\begin{align*}
			\mu\otimes\nu(\varphi(x,y)\triangle\psi(x,y))\leq \eps.
		\end{align*}
	\end{cor}
	
	\begin{proof}
		Let $\lambda=\mu\otimes\nu$. Given $\eps>0$, let $\delta=1-\sqrt{1-\eps}$ and let $\phi_1(x),\ldots,\phi_n(x)$ and $\theta_1(y),\ldots,\theta_m(y)$ be the formulas given by Theorem \ref{regularmeasure} for $\delta$. So
		\begin{align*}
			\Gamma&=\{(i,j):\lambda(\varphi(x,y)\land \phi_i(x)\land\theta_j(y))\geq(1-\delta)^2\cdot\lambda(\phi_i(x)\land\theta_j(y))\}\\
			&=\{(i,j):\lambda(\varphi(x,y)\land\phi_i(x)\land\theta_j(y))\geq(1-\eps)\cdot\lambda(\phi_i(x)\land\theta_j(y))\}.
		\end{align*}
		Observe that for $(i,j)\in \Gamma$ we have that
		\begin{align*}
			\lambda(\neg\varphi(x,y)\land \phi_i(x)\land\theta_j(y))\leq\eps\cdot\lambda(\phi_i(x)\land\theta_j(y)),
		\end{align*}
		while by Theorem \ref{regularmeasure}, for $(i,j)\notin \Gamma$ we have
		\begin{align*}
			\lambda(\varphi(x,y)\land \phi_i(x)\land\theta_j(y))\leq\eps\cdot\lambda(\phi_i(x)\land\theta_j(y)).
		\end{align*}
		Now define
		\begin{align*}
			\psi(x,y)=\bigvee_{(i,j)\in\Gamma}\phi_i(x)\land\theta_j(y).
		\end{align*}
		Clearly $\psi(x,y)$ is stable and
		\begin{align*}
			\lambda(\varphi(x,y)\triangle\psi(x,y))&=\sum_{(i,j)\notin J}\lambda(\varphi(x,y)\land\phi_i(x)\land\theta_j(y))\\&+\sum_{(i,j)\in J}\lambda(\neg\varphi(x,y)\land\phi_i(x)\land\theta_j(y))\\
			&\leq\sum_{i=1}^n\sum_{j=1}^m\eps\cdot\lambda(\phi_i(x)\land\theta_j(y))=\eps,
		\end{align*}
		obtaining the desired bound.
	\end{proof}
	
	\section{Groups and stabilizers}\label{groups}
	
	In this section we work in the presence of an ambient definable group $G$. More precisely and for simplicity, assume that $\MB=(G,\cdot,\ldots)$ in such a way that $(G,\cdot)$ is a group. As usual, if $g$ is an element of the group and $\psi(x)$ is a formula, then $g\cdot\psi(x)$ is defined to be the formula $\psi(g\inv\cdot x)$. Hence, for any partial type $\pi(x)$, we define
	\begin{align*}
		g\cdot\pi(x)=\{\psi(g\inv\cdot x):\psi(x)\in\pi(x)\}.
	\end{align*}
	Henceforth, we are interested in the study of relations of the form
	\begin{align*}
		\{(g,h)\in G\times G:h\cdot g\in A\}
	\end{align*}
	to obtain results that will be applied in Section \ref{applications} to the usual Cayley relation.
	
	Therefore, let $\ideal_x$ and $\ideal_y$ be two $\emptyset$-invariant $S_1$-ideals that are formed by the same formulas up to relabelling the variables $x$ and $y$ that are, in addition, invariant under left and right translations by elements from $G$. As in Section \ref{Istab}, denote by $\ideal$ the ordered pair $(\ideal_x,\ideal_y)$, and consider an $\ideal$-stable formula $\varphi(x,y)=A(y\cdot x)$, where $A$ is an $\emptyset$-definable subset of $G$. In this particular setting, the formula $g\cdot\varphi(x,b)$ is equivalent to $\varphi(x,b\cdot g^{-1})$ for any parameter $b$ and every $g\in G$. We abbreviate this by saying that $\varphi(x,y)$ is \emph{equivariant}.
	
	Since $\ideal_x$ and $\ideal_y$ are now formed essentially by the same formulas, we will not say that a tuple or a type is $\ideal_x$-wide or $\ideal_y$-wide, but just wide. Moreover, note that there is no ambiguity in writing $B\notin\ideal$ to denote that a set $B$ that is either $L_x(\MB)$-definable or $L_y(\MB)$-definable is wide. We assume that $G\notin\ideal$.

	Let $M$ be an elementary substructure. Our first objective is to show that in the presence of the group $G$, the space $\gl{X}_M=\gl{X}_{\varphi,M}$ is finite.
	
	\begin{lemma}\label{actauto}
		Let $N$ be an $\abs{M}^+$-saturated elementary extension of $M$. The group $G(N)$ acts by homeomorphisms on the space $\gl{X}_M$.
	\end{lemma}
	
	\begin{proof}
		Let us see first that $G(N)$ acts on the space $\gl{X}_N$ by translations, and that such action is by homeomorphisms.
		
		Observe that if $b$ is wide over $N$ and $g\in G(N)$, then $b\cdot g$ is wide over $N$. This, together with the equivariance of $\varphi(x,y)$ and the invariance of the ideals $\ideal_x$ and $\ideal_y$ under translations by elements from $G$, yields that for every $g\in G(N)$ the map $g:\gl{X}_N\to\gl{X}_N$ given by $\gl{p}\mapsto g\cdot\gl{p}$ is well-defined. Moreover, this map is clearly bijective and continuous, since for every $b$ that is wide over $N$,
		\begin{align*}
			g\inv\left([\varphi(x,b)]_N\right)=[\varphi(x,b\cdot g)]_N.
		\end{align*}
		The space $\gl{X}_N$ is compact and Hausdorff, so $g$ is an homeomorphism, as we wanted to prove.
		
		Now, by Remark \ref{spaceshomeo} we have that the map $\rho_{M,N}:\gl{X}_M\to\gl{X}_N$ is an homeomorphism. Hence, the action of $G(N)$ on $\gl{X}_M$ defined as
		\begin{align*}
			g\ast\gl{p}=\rho_{M,N}\inv(g\cdot \rho_{M,N}(\gl{p}))
		\end{align*}
		for every $\gl{p}\in\gl{X}_M$ is clearly an action by homeomorphisms.
	\end{proof}
	
	\begin{remark}
		The action of $G(N)$ on $\gl{X}_M$ of the above lemma is simply
		\begin{align*}
			g\ast\gl{p}=(g\cdot\gl{q})\vert_{\wide{M}},
		\end{align*}
		for every $g\in G(N)$, every $\gl{p}\in\gl{X}_M$ and any wide global $\varphi$-type $\gl{q}$ such that $\gl{q}\vert_{\wide{M}}=\gl{p}$.
	\end{remark}
	
	Recall that a definable subset $B$ of $G$ is \emph{generic} if there exists a finite subset $S$ of $G$ such that $G=S\cdot B$ \cite[Definition 5.14]{HP94}. If $\varphi(x,y)$ is stable and $\mu$ is a Keisler measure which is invariant under left translations, then a $\varphi$-formula is generic if and only if it has positive $\mu$-measure \cite[Theorem 2.3]{CPT18}. We now aim to generalize the notion of generic formulas and obtain an analogous result.
	
	\begin{definition}
		A definable subset $B$ of $G$ is $\ideal$\emph{-left-generic} if there exists a finite subset $S$ of $G$ such that $G\setminus (S\cdot B)\in\ideal$. Similarly, the set $B$ is $\ideal$\emph{-right-generic} if there exists a finite subset $S$ of $G$ such that $G\setminus(B\cdot S)\in\ideal$.
	\end{definition}
	
	\begin{definition}
		A \emph{basic} $\varphi$\emph{-definable set}  is a set defined by a formula of the form
		\begin{align*}
			\bigwedge_{i=1}^m\varphi(x,c_i)\land\bigwedge_{j=1}^n\neg\varphi(x,d_j)
		\end{align*}
		for some parameters $c=(c_1,\ldots,c_m)$ and $d=(d_1,\ldots,d_n)$.
	\end{definition}

	\begin{prop}\label{generics}
		Let $B$ be a basic $\varphi$-definable set. Then $B$ is $\ideal$-left-generic if and only if $B\notin\ideal$.
	\end{prop}
	
	\begin{proof}
		Assume first that $B$ is $\ideal$-left-generic. So, there exists a finite subset $S$ of $G$ such that $G\setminus (S\cdot B)\in\ideal$. If $B\in\ideal$ then $S\cdot B\in\ideal$, so $G=G\setminus(S\cdot B)\cup(S\cdot B)\in\ideal$, which is a contradiction. Therefore, $B\notin \ideal$.
		
		For the other direction of the statement, assume that the set $B$ is defined by the formula
		\begin{align*}
			\bigwedge_{i=1}^m\varphi(x,c_i)\land\bigwedge_{j=1}^n\neg\varphi(x,d_j)
		\end{align*}
		for some $c=(c_1,\ldots,c_m)$ and $d=(d_1,\ldots,d_n)$. Let $\psi(x,y)=y\cdot x\in B$ and let $M^\prime$ be a model containing $c$ and $d$. Then:
		
		\begin{claim*}
			There exists some $k\in\NB$ such that no tuple $(a_1,b_1,\ldots,a_k,b_k)$ that is $M^\prime$-independent satisfies $\psi(a_i,b_j)$ if and only if $i\leq j$.
		\end{claim*}
		
		\begin{claimproof*}
			Let $r\in\NB$ be such that both $\varphi(x,y)$ and $\neg\varphi(x,y)$ are $\ideal$-stable of ladder $r$. Now assume, for the sake of contradiction, that there are sequences as in the statement for arbitrary length $k$. Then, if $k$ is big enough, we know by Ramsey's theorem that there exists an $M^\prime$-independent tuple $(a_1,b_1,\ldots,a_r,b_r)$, and some $p\in\{1,\ldots,m\}$ or some $q\in\{1,\ldots,n\}$, such that either $\varphi(b_j\cdot a_i,c_p)$ holds if and only if $i\leq j$, or $\neg\varphi(b_j\cdot a_i,d_q)$ holds if and only if $i\leq j$. Both cases are analogous, so we may assume that we are in the first one. For every $j=1,\ldots,r$, define $b_j^\prime=c_p\cdot b_j$. Then $(a_1,b_1^\prime,\ldots,a_r,b_r^\prime)$ is an $M^\prime$-independent tuple such that
			\begin{align*}
				\varphi(a_i,b_j^\prime)\Leftrightarrow\varphi(a_i,c_p\cdot b_j)\Leftrightarrow\varphi(b_j\cdot a_i,c_p)\Leftrightarrow i\leq j,
			\end{align*}
			which is a contradiction.
		\end{claimproof*}
		
		Finally, suppose that $B\notin\ideal$ but it is not $\ideal$-left-generic. Then for every finite subset $S$ of $G$ we have that $G\setminus (S\cdot B)\notin \ideal$, so there exists a wide global type $\gl{q}_S$ such that $G\setminus (S\cdot B)\in\gl{q}_S$. Equivalently $S\cdot B\notin\gl{q}_S$, and in terms of the compact space $S_x^{\textrm{w}}(\MB)$ formed by the complete wide global types, we have that $\gl{q}_S\notin \bigcup_{g\in S}[g\cdot B]$. By compactness,
		\begin{align*}
			S_x^{\textrm{w}}(\MB)\not\subset\bigcup_{g\in G}[g\cdot B],
		\end{align*}
		and therefore there is a complete wide global type $\gl{q}$ such that $g\cdot B\notin\gl{q}$ for every $g\in G$. Now we will use the type $\gl{q}$ to define inductively a sequence $(a_i,b_i)_{i\geq 1}$ such that each $a_i$ realizes $\tpp(a_1/M^\prime)$, every finite sequence $(a_1,b_1,\ldots,a_k,b_k)$ is $M^\prime$-independent, and $b_j\cdot a_i\in B$ if and only if $i\leq j$.
		
		Assume that we have already constructed $a_1,b_1,\ldots,a_k,b_k$. Let $a_{k+1}$ be a realization of $\gl{q}\vert_{M^\prime,a_{\leq k},b_{\leq k}}$. By the choice of $\gl{q}$, for every $j\leq n$ we have that $b_j\cdot a_{k+1}\notin B$. Now note that in particular $(a_1,\ldots,a_{k+1})$ is an $M^\prime$-independent tuple of realizations of $\tpp(a_1/M^\prime)$. Since $B\cdot a_1^{-1}\notin\ideal$, it follows by Lemma \ref{wideinter} that
		\begin{align*}
			B\cdot a_1^{-1}\cap\ldots\cap B\cdot a_{k+1}^{-1}\notin\ideal,
		\end{align*}
		and there is some $b_{k+1}$ in the above intersection that is wide over $M^\prime,a_{\leq k+1},b_{\leq k}$. This construction clearly contradicts the claim. 
	\end{proof}
	
	\begin{prop}\label{XMfinite}
		The space $\gl{X}_M$ is finite.
	\end{prop}
	
	\begin{proof}
		By compactness it is enough to prove that all the elements of $\gl{X}_M$ are isolated.
		
		The space $\gl{X}_M$ has finite Cantor-Bendixson rank, so there is at least one isolated point, say $\gl{p}$. Clearly we may assume that $\gl{p}$ is isolated by a clopen set of the form $U=[B]$, where $B$ is a basic $\varphi$-definable set. Let $N$ be an $\abs{M}^+$-saturated elementary extension of $M$. We may further assume, by saturation of $N$ and Remark \ref{resultsforXM}, that the set $B$ is definable over $N$. Obviously $B\notin\ideal$, so by Proposition \ref{generics} there exists a finite subset $S$ of $G$ such that $G\setminus (S\cdot B)\in\ideal$. Realizing in $N$ the type of $S$ over the parameters of $B$, we can assume that $S\subset G(N)$.
		
		Let $\gl{p}^\prime$ be an arbitrary element of $\gl{X}_M$, and let $\gl{q}$ be any wide global $\varphi$-type such that $\gl{q}\vert_{\wide{M}}=\gl{p}^\prime$. Then $G\setminus (S\cdot B)\notin\gl{q}$, and there is some $g\in S$ such that $g\cdot B\in\gl{q}$, or equivalently, $B\in g\inv\cdot\gl{q}$. The parameters of $B$ are wide over $M$, so indeed we have that
		\begin{align*}
			B\in(g\inv\cdot\gl{q})\vert_{\wide{M}}=g\inv\ast\gl{p}^\prime.
		\end{align*}
		Now, since $\gl{p}$ is isolated in $\gl{X}_M$ by $[B]$, it follows that $\gl{p}=g^{-1}\ast\gl{p}^\prime$. As the action of $G(N)$ on $\gl{X}_M$ is by homeomorphisms, we have that the type $\gl{p}^\prime$ is isolated.
	\end{proof}
	
	\begin{cor}\label{stabXM}
		Let $N$ be an $\abs{M}^+$-saturated elementary extension of $M$. For every $\gl{p}\in \gl{X}_M$, the subgroup
		\begin{align*}
			\stab_{G(N)}(\gl{p})=\{g\in G(N):g\ast\gl{p}=\gl{p}\}
		\end{align*}
		is definable over $M$ and has finite index.
	\end{cor}
	
	\begin{proof}
		By Proposition \ref{XMfinite} we have that the space $\gl{X}_M$ is finite, so clearly all the stabilizers of the statement have finite index in $G(N)$. Now recall by the proof of Lemma \ref{actauto} that $G(N)$ also acts on $\gl{X}_N$ by translations. Hence, we will first prove that given $\gl{p}\in\gl{X}_N$, the subgroup
		\begin{align*}
			\stab_{G(N)}(\gl{p})=\{g\in G(N):g\cdot\gl{p}=\gl{p}\}
		\end{align*}
		is definable over $M$. This is enough, because for every $\gl{p}\in\gl{X}_M$ we obtain that $\rho_{M,N}(\gl{p})\in\gl{X}_N$ and
		\begin{align*}
			\stab_{G(N)}(\gl{p})=\stab_{G(N)}(\rho_{M,N}(\gl{p})),
		\end{align*}
		where each of the stabilizers is computed with the corresponding action.
		
		Let $\gl{X}_N=\{\gl{p}_1,\ldots,\gl{p}_n\}$. If $n=1$ then $\stab_{G(N)}(\gl{p}_1)=G(N)$, which is definable over $M$. Therefore, suppose that $n\geq 2$ and assume without loss of generality that $\gl{p}=\gl{p}_1$. For every $i=2,\ldots,n$ let $\eps_i\in\{0,1\}$ and let $b_i$ be a wide tuple over $N$ such that $\varphi(x,b_i)^{\eps_i}\in\gl{p}$ and $\neg\varphi(x,b_i)^{\eps_i}\in\gl{p}_i$.
		
		\begin{claim}\label{stab1}
			Let $g\in G(N)$. Then $g\in\stab_{G(N)}(\gl{p})$ if and only if $\varphi(x,b_i)^{\eps_i}\in g\cdot\gl{p}$ for every $i=2,\ldots n$.
		\end{claim}
		
		\begin{claimproof}
			It is obvious that if $g\in\stab_{G(N)}(\gl{p})$ then $\varphi(x,b_i)^{\eps_i}\in\gl{p}= g\cdot\gl{p}$ for every $i=2,\ldots,n$. Reciprocally, suppose that $g\notin\stab_{G(N)}(\gl{p})$, so $g\cdot\gl{p}=\gl{p}_j$ for some $j\neq 1$. Since $\varphi(x,b_j)^{\eps_j}\notin\gl{p}_j$, then $\varphi(x,b_j)^{\eps_j}\notin g\cdot\gl{p}$.
		\end{claimproof}
		
		\begin{claim}\label{stab2}
			Let $\Phi(y)$ be the partial type over $M$ asserting that $y$ is wide over $M$. Then for every wide complete type $p(y)\in S_y(M)$ and every wide global $\varphi$-type $\gl{q}(x)$,
			\begin{align*}
				p(y)\land p(y^\prime)\land\Phi(y\cdot z)\land\Phi(y^\prime\cdot z)\vdash\deftpth{\gl{q},M}{\varphi(y\cdot z)}\leftrightarrow\deftpth{\gl{q},M}{\varphi(y^\prime\cdot z)}.
			\end{align*}
		\end{claim}
		
		\begin{claimproof}
			Let $g\in G$ and assume that $c$ and $c^\prime$ are realizations of $p(y)$ such that both $c\cdot g$ and $c^\prime\cdot g$ are wide over $M$. As $g\cdot\gl{q}$ is a wide global $\varphi$-type and both $c$ and $c^\prime$ are also wide over $M$, by Lemma \ref{inv} and Theorem \ref{def2} we obtain,
			\begin{align*}
				\deftp{\gl{q},M}{\varphi(c\cdot g)}\Leftrightarrow\varphi(x,c\cdot g)\in\gl{q}&\Leftrightarrow\varphi(x,c)\in g\cdot\gl{q}\Leftrightarrow\varphi(x,c^\prime)\in g\cdot\gl{q}\\
				&\Leftrightarrow\varphi(x,c^\prime\cdot g)\in\gl{q}
				\Leftrightarrow\deftp{\gl{q},M}{\varphi(c^\prime\cdot g)},
			\end{align*}
			as required.
		\end{claimproof}
		
		For every $i=2,\ldots,n$ let $p_i(y)=\tpp(b_i/M)$. Applying Claim \ref{stab2} to the types $p_i(y)$ and any wide global $\varphi$-type $\gl{q}(x)$ such that $\gl{q}\vert_{\wide{N}}=\gl{p}$, we obtain by compactness $L(M)$-formulas $\theta_i(y)\in p_i(y)$ and $\psi(y)\in\Phi(y)$ such that
		\begin{align*}
			\theta_i(y)\land\theta_i(y^\prime)\land\psi(y\cdot z)\land\psi(y^\prime\cdot z)\vdash\deftp{\gl{q},M}{\varphi(y\cdot z)}\leftrightarrow\deftp{\gl{q},M}{\varphi(y^\prime\cdot z)}.
		\end{align*}
		Let us see that $\stab_{G(N)}(\gl{p})$ is defined over $M$ by the formula
		\begin{align*}
			\phi(z)=\exists y_2,\ldots,y_n\left(\bigwedge_{i=2}^n\theta_i(y_i)\land\psi(y_i\cdot z)\land\deftp{\gl{q},M}{\varphi^{\eps_i}(y_i\cdot z)}\right).
		\end{align*}
		Observe that for every $g$ in $G(N)$, each $\theta_i(b_i)\land\psi(b_i\cdot g)$ holds. Now suppose that $g\in\stab_{G(N)}(\gl{p})$. By Claim \ref{stab1} we have that $\deftp{\gl{q},M}{\varphi(b_i\cdot g)^{\eps_i}}$ holds. Let $c_i\equiv_Mb_i$ be a tuple in $N$ that is wide over $M,g$. Then clearly $\theta_i(c_i)\land\psi(c_i\cdot g)$ holds, which implies that $\deftp{\gl{q},M}{\varphi(c_i\cdot g)^{\eps_i}}$ holds. So $g$ satisfies $\phi(z)$. On the other hand, assume that $g\in G(N)$ satisfies $\phi(z)$ witnessed by some $c_2,\ldots,c_n$ in $N$. By Claim \ref{stab2}, every $\deftp{\gl{q},M}{\varphi(b_i\cdot g)^{\eps_i}}$ holds, so by Claim \ref{stab1} we obtain that $g\in\stab_{G(N)}(\gl{p})$.
	\end{proof}

	We are now ready to prove one of the key results of this section, which is an adaptation of \cite[Proposition 3.1]{AMP24} to our $\ideal$-stable context. Recall that at the beginning of the section we fixed an  $\emptyset$-definable set $A$ and that we are working with the $\ideal$-stable formula $\varphi(x,y)=A(y\cdot x)$.
	
	\begin{lemma}\label{cosets&wide}
		There exists an $M$-definable finite-index normal subgroup $H$ of $G$ with the following property: for every left coset $C$ of $H$ such that $C\cap A$ is wide, we have that $C\setminus A\in\ideal$.
	\end{lemma}
	
	\begin{proof}
		Let $N$ be an $\abs{M}^+$-saturated elementary extension of $M$ and let
		\begin{align*}
			H(N)=\bigcap_{\gl{p}\in\gl{X}_M}\stab_{G(N)}(\gl{p}).
		\end{align*}
		Let $H(x)$ be a formula over $M$ defining this intersection, given by Corollary \ref{stabXM}. The subgroup $H(N)$ of $G(N)$ is normal and has finite index, and so does $H=H(\MB)$ in $G$. 
		
		Let $C=g\cdot H$ for some $g\in G$. If $C\setminus A\notin\ideal$ then there exists an element $h\in C\setminus A$ that is wide over $M,g$. Such element is in particular wide over $M$, so it is enough to prove that every $h\in C$ that is wide over $M$ belongs to $A$.
		
		Since $C\cap A$ is wide, we may assume that $g$ belongs to $A$ and is wide over $M$. Let $h\in C$ be wide over $M$. Realizing the types $\tpp(g/M)$ and $\tpp(h/M)$ in $N$ we may also assume that $g,h\in G(N)$. Now let $b\in N$ be a tuple that is wide over $M,g$ and has the same type over $M$ as $h$. As $H(N)$ has finite index in $G(N)$ then $g\cdot b\inv\in H(N)$, and additionally we have that both $g\cdot b\inv$ and $b\cdot g\inv$ are wide over $M,g$.
		
		The type $\tpp(g/M)$ is wide, so it can be extended to a wide global type $\gl{q}^\prime$. Denote by $\gl{q}=\gl{q}^\prime\vert_\varphi$, which is a wide global $\varphi$-type. Since $g\cdot b^{-1}\in H(N)$ we obtain
		\begin{align*}
			(g\cdot b\inv\cdot\gl{q})\vert_{\wide{M}}=\gl{q}\vert_{\wide{M}}.
		\end{align*}
		On the other hand, since $g\in A$ and $A$ is defined by the formula $\varphi(x,1)$, we have that $\varphi(x,1)\in\gl{q}$, which implies that $\varphi(x,b\cdot g\inv)\in g\cdot b\inv\cdot \gl{q}$. As noted before, $b\cdot g\inv$ is wide over $M$, so in fact $\varphi(x,b\cdot g\inv)\in (g\cdot b\inv\cdot\gl{q})\vert_{\wide{M}}=\gl{q}\vert_{\wide{M}}$.
		
		Finally, let $a$ be a realization of $\gl{q}^\prime\vert_{M,b,g}$. Then $\varphi(a,b\cdot g\inv)$ holds, $a\equiv_Mg$, and both $(b\cdot g\inv,a)$ and $(g,b\cdot g\inv)$ are $M$-independent tuples. By Theorem \ref{stat}, $\varphi(g,b\cdot g\inv)$ holds, which means that $b\in A$. Since $h\equiv_Mb$, we obtain that $h\in A$.
	\end{proof}
	
	The main consequence of the above theorem is that every $\ideal$-stable set is \emph{comparable} to a finite union of cosets of a finite-index subgroup. It is immediate to check that such a finite union $B$ is stable in the sense that the formula $y\cdot x\in B$ is stable.
	
	\begin{theorem}\label{almoststable}
		There exists an $M$-definable finite-index normal subgroup $H$ of $G$ and a finite union $B$ of cosets of $H$ such that
		\begin{align*}
			A(x)\triangle B(x)\in\ideal.
		\end{align*}
	\end{theorem}
	
	\begin{proof}
		Let $H$ be given by Lemma \ref{cosets&wide} and define the set
		\begin{align*}
			B=\bigcup\{C:C\textrm{ is a left coset of }H\textrm{ such that }C\cap A\textrm{ is wide}\}.
		\end{align*}
		The subgroup $H$ has finite index in $G$, so $B$ is a finite union of cosets, and hence, stable. By construction, $A\setminus B\in\ideal$. On the other hand, we have by Theorem \ref{cosets&wide} that $C\setminus A\in\ideal$ for every left coset $C$ of $H$ such that $C\cap A$ is wide. Consequently $B\setminus A\in\ideal$, and we can conclude that $A\triangle B\in\ideal$.
	\end{proof}
	
	\section{Almost stable sets in finite graphs and finite groups}\label{applications}
	
	In this final part we apply the results from the previous two sections to prove a graph regularity lemma and an arithmetic regularity lemma for arbitrary finite groups.
	
	\begin{definition} \cite[Definition 3.19]{SS17}
		A Keisler measure $\mu=\mu_x$ is \emph{definable over} $A$ if for every formula $\psi(x,y)\in L$ and every $\eps>0$, there is a partition of $\MB^{y}$ into $L(A)$-formulas $\rho_1(y),\ldots,\rho_m(y)$ such that for every $i=1,\ldots,m$ and every $b$ and $b^\prime$ realizing $\rho_i(y)$,
		\begin{align*}
			\abs{\mu(\psi(x,b))-\mu(\psi(x,b^\prime))}<\eps.
		\end{align*}
	\end{definition}
	
	\begin{remark}
		If $\mu$ is definable over $A$, then it is definable over every model $N$ containing $A$, and for every $L(N)$-formula $\psi(x,y)$ the function
		\begin{align*}
			F_{\mu,N}^\psi:S_y(M)\to[0,1]\textrm{, }\tpp(b/N)\mapsto\mu(\psi(x,b))
		\end{align*}
		is well-defined and continuous \cite[Proposition 2.17]{KG20}. In particular, every definable measure is Borel-definable over the same set of parameters. Moreover, if $\mu=\mu_x$ and $\nu=\nu_y$ are definable over $N$ and $\eta=\eta_z$ is any Keisler measure, then $\mu\otimes\nu$ is definable over $N$ and $\mu\otimes(\nu\otimes\eta)=(\mu\otimes\nu)\otimes\eta$ \cite[Proposition 2.24]{KG20}.
	\end{remark}
	
	\begin{fact}\cite[Lemmas 1.10 and 1.15]{AMP24}\label{widetuple}
		Let $\mu=\mu_x$ and $\nu=\nu_y$ be two $\emptyset$-definable Keisler measures satisfying Fubini-Tonelli, i.e. for every formula $\psi(x,y)$ and every model $N$ containing the parameters of $\psi$,
		\begin{align*}
			\int_{\tpp(a/N)\in S_x(N)}\nu(\psi(a,y))\mathrm{d}\mu_x\vert_N=\int_{\tpp(b/N)\in S_y(N)}\mu(\psi(x,b))\mathrm{d}\nu_y\vert_N.
		\end{align*}
		Then a formula $\varphi(x,y)$ is almost sure stable of ladder $k$ with respect to $\mu$ and $\nu$ if and only if
		\begin{align*}
			\lambda_k\left(\{(a_1,b_1,\ldots,a_k,b_k):\varphi(a_i,b_j)\textrm{ holds if and only if }i\leq j\}\right)=0,
		\end{align*}
		where $\lambda_k=\mu\otimes\nu\otimes\overset{(k)}{\cdots}\otimes\mu\otimes\nu$.
	\end{fact}
	
	We now proceed to set the context that will be used throughout the results of this section.
	
	Let $L$ be a countable language with sorts $U(x)$ and $V(y)$, and for every natural number $n$ consider a finite $L$-structure $G_n=(U_n,V_n,\ldots)$ with $\abs{U_n},\abs{V_n}\geq\abs{n}$. Let $\mcl{U}$ be a non-principal ultrafilter on $\NB$ and define the ultraproduct $G=\prod_{n\to\mcl{U}}G_n$, which is an infinite $\aleph_1$-saturated structure equipped, among others, with the internal sets $U=\prod_{n\to\mcl{U}}U_n$ and $V=\prod_{n\to\mcl{U}}V_n$ corresponding with the sorts of $x$ and $y$. In each $G_n$ we have the normalized counting measures $\mu_n$ and $\nu_n$ given by
	\begin{align*}
		\mu_n(A)=\frac{\abs{A}}{\abs{U_n}}\textrm{ for }A\subset U_n\textrm{ and }
		\nu_n(B)=\frac{\abs{B}}{\abs{V_n}}\textrm{ for }B\subset V_n.
	\end{align*}
	Then we can assume that these measures induce $\emptyset$-definable global Keisler measures $\mu=\mu_x$ and $\nu=\nu_y$ on $T=\textrm{Th}(G)$ with the additional property that for all internal sets $A\subset U$ and $B\subset V$,
	\begin{align*}
		\mu(A)=\lim_{n\to\mcl{U}}\mu_n(A(G_n))\textrm{ and }
		\nu(B)=\lim_{n\to\mcl{U}}\nu_n(B(G_n)),
	\end{align*}
	where $A(G_n)$ is the trace of the internal set $A$ in $G_n$ (see \cite[Section 2.6]{EH12}).
	
	\begin{lemma}\label{fubini}
		The measures $\mu$ and $\nu$ satisfy Fubini-Tonelli.
	\end{lemma}
	
	\begin{proof}
		Define the global measures $\lambda_1$ and $\lambda_2$ as follows. For every formula $\psi(x,y)$,
		\begin{align*}
			\lambda_1(\psi(x,y))=\int_{\tpp(a/N)\in S_x(N)}\nu(\psi(a,y))\mathrm{d}\mu_x\vert_N
		\end{align*}
		and
		\begin{align*}
			\lambda_2(\psi(x,y))=\int_{\tpp(b/N)\in S_y(N)}\mu(\psi(x,b))\mathrm{d}\nu_y\vert_N,
		\end{align*}
		where $N$ is any model of $T$ containing the parameters of $\psi(x,y)$. If $\mu$ and $\nu$ do not satisfy Fubini-Tonelli, then there is some formula $\psi(x,y,z)\in L$ such that $\lambda_1(\psi(x,y,c))\neq\lambda_2(\psi(x,y,c))$, where $c$ is a parameter that belongs to some model of $T$. Since $G$ is $\aleph_1$-saturated, we can realize $\tpp(c)$ by some $c^\prime$ in $G$, and by the invariance of the measures $\lambda_1$ and $\lambda_2$ we obtain that $\lambda_1(\psi(x,y,c^\prime))\neq\lambda_2(\psi(x,y,c^\prime))$. But note that
		\begin{align*}
			\lambda_1(\psi(x,y,c^\prime))=\int_{\tpp(a/G)\in S_x(G)}\nu(\psi(a,y,c^\prime))\mathrm{d}\mu_x\vert_G
		\end{align*}
		and
		\begin{align*}
			\lambda_2(\psi(x,y,c^\prime))=\int_{\tpp(b/G)\in S_y(G)}\mu(\psi(x,b,c^\prime))\mathrm{d}\nu_y\vert_G.
		\end{align*}
		This is a contradiction, for $\psi(x,y,c^\prime)$ is an $L(G)$-formula and hence the above integrals are both equal to the ultralimit
		\begin{align*}
			\lim_{n\to\mcl{U}}\frac{\abs{\psi(G_n,c^\prime)}}{\abs{U_n\times V_n}}
		\end{align*}
		by \cite[Theorem 19]{VT13} and its proof, where $\psi(G_n,c^\prime)$ is the set of pairs $(a,b)\in U_n\times V_n$ such that $\psi(a,b,c^\prime)$ holds in $G_n$. Therefore $\mu$ and $\nu$ satisfy Fubini-Tonelli.
	\end{proof}
	
	We now prove that if a formula $\varphi(x,y)\in L$ has only a few counterexamples to stability in each $G_n$, then it is almost sure stable in the ultraproduct. More precisely:
	
	\begin{definition}
		Given a formula $\varphi(x,y)\in L$ and $k\in\NB$, define
		\begin{align*}
			\mcl{H}_k(\varphi(G_n))=\{(a_1,b_1,\ldots,a_k,b_k)\in (U_n\times V_n)^k:\varphi(a_i,b_j)\textrm{ holds}\Leftrightarrow i\leq j\}
		\end{align*}
		to be the set of half-graphs of height $k$ induced by $\varphi(x,y)$.
	\end{definition}
	
	\begin{lemma}\label{finitestr}
		Let $\varphi(x,y)\in L$. If there exists a natural number $k$ such that
		\begin{align*}
			\frac{\abs{\mcl{H}_k(\varphi(G_n))}}{\abs{U_n\times V_n}^k}<\frac{1}{n}
		\end{align*}
		for every natural number $n$, then the formula $\varphi(x,y)$ is almost sure stable of ladder $k$ with respect to $\mu$ and $\nu$ in the theory $T$.
	\end{lemma}
	
	\begin{proof}
		By Lemma \ref{fubini}, the measures $\mu$ and $\nu$ satisfy Fubini-Tonelli, so if $\varphi(x,y)$ is not almost sure stable of ladder $k$ with respect to $\mu$ and $\nu$, then by Fact \ref{widetuple} the formula
		\begin{align*}
			\psi(x_1,y_1,\ldots,x_k,y_k)=\bigwedge_{i\leq j}\varphi(x_i,y_j)\land\bigwedge_{i>j}\neg\varphi(x_i,y_j)
		\end{align*}
		has positive measure, say $\lambda_k(\psi)>1/N$ for some natural number $N$. On the other hand, by \cite[Theorem 19]{VT13} we have that
		\begin{align*}
			\lambda_k(\psi)=\lim_{n\to\mcl{U}}\frac{\abs{\psi(G_n)}}{\abs{U_n\times V_n}^k}=\frac{\abs{\mcl{H}_k(\varphi(G_n))}}{\abs{U_n\times V_n}^k},
		\end{align*}
		implying that there exists some $n\geq N$ such that $\abs{\mcl{H}_k(\varphi(G_n))}/\abs{U_n\times V_n}^k>1/n$. By assumption this is a contradiction, so $\varphi(x,y)$ is almost sure stable of ladder $k$ with respect to $\mu$ and $\nu$.
	\end{proof}
	
	Let $G=(U,V,E)$ be a bipartite graph whose vertices are partitioned into $U$ and $V$ and whose edges are $E\subset U\times V$. Recall that given $X\subset U$ and $Y\subset V$, we say that $E\vert_{X\times Y}$ is $\eps$\emph{-homogeneous} if either
	\begin{align*}
		\abs{E\cap(X\times Y)}\geq(1-\eps)^2\cdot\abs{X\times Y}\textrm{ or }\abs{\neg E\cap(X\times Y)}\geq(1-\eps)^2\cdot\abs{X\times Y}.
	\end{align*}
	In the next result we show that if the edge relation $E$ induces very few half-graphs of height $k$, then homogeneous partitions of the sets $U$ and $V$ can be obtained, yet without the explicit polynomial bounds that can be deduced from the proof of \cite[Lemma 6.14]{TW21}. However, since our approach is model-theoretic rather than combinatorial, the proof provides structural control over the pieces of the resulting partitions, ensuring that the obtained sets have bounded complexity.
	
	\begin{theorem}\label{fgraphs}
		For every natural number $k$ and all real numbers $\eps,\delta>0$, there are natural numbers $n=n(k,\eps,\delta)$, $r=r(k,\eps,\delta)$ and $s=s(k,\eps,\delta)$ with the following property: if $G=(U,V,E)$ is any finite bipartite graph with $\abs{U},\abs{V}\geq n$ and $\abs{\mcl{H}_k(E)}<\abs{U}^k\abs{V}^k/n$, then we can partition $U$ into sets $U_0,U_1,\ldots,U_r$ such that each $U_i$ with $i\geq 1$ is a Boolean combination of neighbourhoods of the form $E(x,b)$ and $\neg E(x,b)$ of complexity at most $n$ with $\abs{U_i}\geq\abs{U}/n$, and $V$ into sets $V_0,V_1,\ldots,V_s$ such that each $V_i$ with $i\geq 1$ is a Boolean combination of neighbourhoods of the form $E(a,y)$ and $\neg E(a,y)$ of complexity at most $n$ with $\abs{V_i}\geq\abs{V}/n$, such that
		\begin{align*}
			\abs{U_0}<\delta\cdot\abs{U} \textrm{ and } \abs{V_0}<\delta\cdot\abs{V}
		\end{align*}
		and for every $i,j\geq 1$, the relation $E\vert_{U_i\times V_j}$ is $\eps$-homogeneous.
	\end{theorem}
	
	\begin{proof}
		If the statement is false, then there is a natural number $k$ and real numbers $\eps,\delta>0$ such that, for all natural numbers $n,r,s$ there is a finite bipartite graph $G_n=(U_n,V_n,E_n)$ with $\abs{U_n},\abs{V_n}\geq\abs{n}$ and $\abs{\mcl{H}_k(E_n)}<\abs{U_n}^k\abs{V_n}^k/n$ with the following property: for all pairwise disjoint subsets $U_{n,1},\ldots,U_{n,r}$ of $U_n$ and $V_{n,1},\ldots,V_{n,s}$ of $V_n$ such that each $U_{n,i}$ is a Boolean combination of neighbourhoods of the form $E(x,b)$ and $\neg E(x,b)$ of complexity at most $n$ with $\abs{U_{n,i}}\geq\abs{U_n}/n$, and such that each $V_{n,i}$ is a Boolean combination of neighbourhoods of the form $E(a,y)$ and $\neg E(a,y)$ of complexity at most $n$ with $\abs{V_{n,i}}\geq\abs{V_n}/n$, if
		\begin{align*}
			\abs{U_{n,1}\cup\ldots\cup U_{n,r}}>(1-\delta)\cdot\abs{U_n} \textrm{ and } \abs{V_{n,1}\cup\ldots\cup V_{n,s}}>(1-\delta)\cdot\abs{V_n},
		\end{align*}
		then there are $1\leq i\leq r$ and $1\leq j\leq s$ such that $E_n\vert_{U_{n,i}\times V_{n,j}}$ is not $\eps$-homogeneous.
		
		The graphs $G_n$ are $L$-structures, where $L$ is a suitable extension of the language of bipartite graphs, see \cite[Section 2.6]{EH12}. Let $\mcl{U}$ be a non-principal ultrafilter on $\NB$, and define $G$ to be the ultraproduct $\prod_{n\to\mcl{U}}G_n$. By Lemma \ref{finitestr}, the formula $E(x,y)$ is almost sure stable of ladder $k$ in the theory $T=\textrm{Th}(G)$ with respect to $\emptyset$-definable global measures $\mu$ and $\nu$ induced by the counting measures $\mu_n$ and $\nu_n$ defined, respectively, on each $U_n$ and each $V_n$.

		Fix a countable elementary substructure $M$ of $G$, and consider the topological spaces $\gl{X}_{E,M}$ and $\gl{X}_{E^*,M}$ defined in Section \ref{topology}. By Theorem \ref{regularmeasure}, we can find pairwise disjoint clopen subsets $U_1,\ldots,U_r$ of $\gl{X}_{E,M}$ and pairwise disjoint clopen subsets $V_1,\ldots,V_s$ of $\gl{X}_{E^*,M}$ such that:
		\begin{itemize}
			\item $\mu(U_1\cup\ldots\cup U_r)=\nu(V_1\cup\ldots\cup V_s)=1$,
			\item each $U_i$ is defined by a Boolean combination of formulas of the form $E(x,b)$ and $\neg E(x,b)$,
			\item each $V_j$ is defined by a Boolean combination of formulas of the form $E(a,y)$ and $\neg E(a,y)$,
			\item for every $1\leq i\leq r$ and $1\leq j\leq s$, the relation $E\vert_{U_i\times V_j}$ is $\eps$-homogeneous with respect to the measure $\lambda=\mu\otimes\nu$.
		\end{itemize}
		Moreover, every $U_i$ has positive $\mu$-measure and every $V_i$ has positive $\nu$-measure, so there exists a natural number $n_1$ such that $\mu(U_i),\nu(V_i)>1/n_1$ for every $i\geq 1$. Now let $n_2$ be the maximum of the complexities of the Boolean combinations defining each $U_i$ and each $V_i$.
		
		By saturation of $G$ and invariance of $\lambda$, we may assume that the sets $U_i$ and $V_i$ are definable over $G$. Hence, for $\mcl{U}$-almost every natural number $n$, we have that the sets $U_1(G_n),\ldots,U_r(G_n)$ are pairwise disjoint, the sets $V_1(G_n),\ldots,V_s(G_n)$ are pairwise disjoint, the relation $E_n\vert_{U_i(G_n)\times V_j(G_n)}$ is $\eps$-homogeneous for every $i$ and $j$, and $\mu_n(U_1(G_n)\cup\ldots\cup U_r(G_n))>1-\delta$ and $\nu_n(V_1(G_n)\cup\ldots\cup V_s(G_n))>1-\delta$. These two latter conditions are equivalent to
		\begin{align*}
			\abs{U_1(G_n)\cup\ldots\cup U_r(G_n)}>(1-\delta)\cdot\abs{U_n} 
		\end{align*}
		and
		\begin{align*}
			\abs{V_1(G_n)\cup\ldots\cup V_s(G_n)}>(1-\delta)\cdot\abs{V_n}.
		\end{align*}
		Finally, note that for $n>\max\{n_1,n_2\}$, the complexities of the Boolean combinations defining each $U_i(G_n)$ and each $V_i(G_n)$ are at most $n$, and $\abs{U_i(G_n)}\geq\abs{U_n}/n$ and $\abs{V_i(G_n)}\geq\abs{V_n}/n$. This contradicts the first paragraph of the proof.
	\end{proof}
	
	\begin{remark}
		Observe that the homogeneity parameter $\eps>0$ in Theorem \ref{fgraphs} is not a function of the number of pieces of the partitions, but rather a constant fixed beforehand. Thus, since $U_0$ and $V_0$ are of constant proportion, if we drop the condition that each $U_i$ and each $V_i$ is a Boolean combination of neighbourhoods of the form $E(x,b)$ and $\neg E(x,b)$, and $E(a,y)$ and $\neg E(a,y)$, respectively, then we can proportionally redistribute the sets $U_0$ and $V_0$ among the other pieces of the partitions without ruining homogeneity (see the second paragraph after Theorem 2.1.2 in \cite{CT23}). While this would change the exact level of homogeneity, it would yield completely homogeneous partitions, like in stable graphs. Indeed, the actual difference between stable graphs and graphs with few witnesses to instability has to do with a more subtle change in the level of \emph{goodness} present in the partitions, as it can be clearly seen in the statement of \cite[Lemma 6.14]{TW21}.
	\end{remark}

	Finally, we prove a regularity lemma for almost stable subsets of groups, extending \cite[Theorem 5.25]{TW23} to arbitrary finite groups, yet without providing effective bounds. Given a subset $A$ of a group $G$, define
	\begin{align*}
		\textrm{Cay}(G,A)=\{(g,h)\in G\times G:g^{-1}\cdot h\in A\}.
	\end{align*}
	
	\begin{theorem}\label{fgroups}
		For every natural number $k$ and every real number $\eps>0$, there exists a natural number $n=n(k,\eps)$ such that: for every finite group $G$ and every subset $A$ of $G$, if $\abs{\mcl{H}_k(\emph{Cay}(G,A))}<\abs{G}^{2k}/n$, then there exists a normal subgroup $H$ of $G$ with index at most $n$, and a set $B$ that is union of cosets of $H$ such that $\abs{A\triangle B}<\eps\cdot\abs{H}$.
	\end{theorem}
	
	\begin{proof}
		If the theorem is false, then there exists a natural number $k$ and a real number $\eps>0$ for which there is a family of finite groups $(G_n)_{n\in\NB}$, each equipped with a distinguished subset $A_n$, such that $\abs{\mcl{H}_k(\textrm{Cay}(G_n,A_n))}/\abs{G_n}^{2k}<1/n$, and for every normal subgroup $H$ of $G_n$ with index $[G_n:H]\leq n$ and any set $B$ that is union of cosets of $H$, we have $\abs{A_n\triangle B}\geq \eps\cdot\abs{H}$. Observe that
		\begin{align*}
			\frac{\abs{A_n\triangle B}}{\abs{G_n}}\geq\frac{\eps}{[G_n:H]}.
		\end{align*}
		Note also that $\abs{G_n}>n$ for every $n$, for otherwise $H$ can be taken to be the trivial group, which has index $n$ in $G_n$, and $A_n=B$.
		
		As in the proof of Theorem \ref{fgraphs}, we will work in a suitable countable language $L$ expanding the language of groups with a predicate $A(x)$ for the sets $A_n$. Regarding the setting described at the beginning of this section, we are now in the particular case where the sets $U_n$ and $V_n$ corresponding with the sorts $U(x)$ and $V(y)$ are both equal to $G_n$.
		\begin{claim*}
			Let $\varphi(x,y)=A(y\cdot x)$. Then $\abs{\mcl{H}_k(\emph{Cay}(G_n,A_n))}=\abs{\mcl{H}_k(\varphi(G_n))}$.
		\end{claim*}
		
		\begin{claimproof*}
			A tuple $(a_1,b_1,\ldots,a_k,b_k)$ belongs to $\mcl{H}_k(\textrm{Cay}(G_n,A_n))$ if and only if the tuple $(c_1,d_1,\ldots,c_k,d_k)$ defined as $c_i=b_{k-i+1}$ and $d_i=a^{-1}_{k-i+1}$ belongs to $\mcl{H}_k(\varphi(M_n))$. In other words, there is a bijection from $\mcl{H}_k(\textrm{Cay}(G_n,A_n))$ into $\mcl{H}_k(\varphi(G_n))$, as we wanted to prove.
		\end{claimproof*}
		Let $\mcl{U}$ be a non-principal ultrafilter on $\NB$, and let $G=\prod_{n\to\mcl{U}}G_n$. By the claim and Lemma \ref{finitestr}, the formula $A(y\cdot x)$ is almost sure stable of ladder $k$ in the theory $T=\textrm{Th}(G)$ with respect to $\mu$ in both variables $x$ and $y$, where $\mu$ is the $\emptyset$-definable global measure induced by the counting measures $\mu_n$ that we have on every $G_n$.
		
		By Theorem $\ref{almoststable}$ and saturation of $G$, there exists a definable normal subgroup $H$ of $G$ with finite index, say $N$, and a set $B$ that is union of cosets of $H$, such that $\mu(A\triangle B)=0$. Then there is some $n\geq N$ such that $H(G_n)$ is a normal subgroup of $G_n$ with index $N\leq n$, the set $B(G_n)$ is union of cosets of $H(G_n)$ and
		\begin{align*}
			\frac{\abs{A(G_n)\triangle B(G_n)}}{\abs{G_n}}<\frac{\eps}{N}=\frac{\eps}{[G_n:H(G_n)]},
		\end{align*}
		which is a contradiction. Hence, the theorem holds.
	\end{proof}
	
	\begin{remark}\label{compTW}
		Our stability assumption is stronger than the one of \cite[Theorem 5.25]{TW23}, where one only bounds trees with leaves in a fixed subgroup of $G$. This is equivalent to bounding the number of half-graphs with one side in that subgroup (see Theorem \ref{tree2}). In particular, when this fixed subgroup is the whole group $G$, then Theorem \ref{fgroups} corresponds to the case where we take $\Omega=\emptyset$ in \cite[Theorem 5.25]{TW23}. In that case, their theorem implies that $\abs{A\triangle B}<\eps\cdot\abs{G}$, which is not as strong as the conclusion provided by Theorem \ref{fgroups}, namely, that $\abs{A\triangle B}<\eps\cdot\abs{H}$.
	\end{remark}
	
	\appendix
	
	\section*{Appendix}\label{apptrees}
	
	\setcounter{section}{1}
	\setcounter{theorem}{0}
	\counterwithin{theorem}{section}
	\renewcommand{\thetheorem}{A.\arabic{theorem}}

	We now revisit a well-known result by Hodges, namely \cite[Lemma 6.7.9]{WH93}. We follow the lines of Hodges's proof, but we use the notion of subtree from \cite{NA22} (where, indeed, a combinatorial proof of the same fact can be found \cite[Theorem 10]{NA22}). This notion eases the verification of certain steps in Hodges's original proof, and therefore we have included it here for the sake of completeness.
	
	\begin{notation}
		We will use the following notation: $^n2$ is the set of all functions of the form $\sigma:\{0,\ldots,n-1\}\to\{0,1\}$, $^{<n}2=\bigcup_{i<n}{^i2}$ and $^{<\omega}2=\bigcup_{n\in\NB}{^n2}$. As usual, we identify the functions of $^{n}2$ with finite sequences whose terms are either $0$ or $1$. If $\sigma\in{^n2}$ and $j\in\{0,1\}$, then $\sigma^\smallfrown j\in{^{n+1}2}$ is the sequence $(\sigma(0),\ldots,\sigma(n-1),j)$, and if $j\leq n$, then $\sigma\vert j\in{^j2}$ is the initial segment of $\sigma$ of length $j$.
	\end{notation}
	
	\begin{definition}\label{bintree}
		Let $\varphi(x,y)$ be a formula in $L$. An $n$\emph{-tree} or a \emph{tree of height} $n$ for $\varphi$ consists of two families of tuples, $(a_\sigma:\sigma\in{^n2})$ and $(b_\tau:\tau\in{^{<n}2})$, such that
		\begin{align*}
			\varphi(a_\sigma,b_{\sigma\vert i})\textrm{ holds if and only if }\sigma(i)=0.
		\end{align*}
		The tuples $a_\sigma$ are the \emph{branches} of the tree, while the tuples $b_\tau$ are its \emph{nodes}. The nodes $b_\tau$ with $\tau\in{^{n-1}2}$ (that is, the nodes of the ``lowest'' level of the tree) are called \emph{leaves}. An \emph{internal node} is a node which is not a leaf. If we fix an internal node $b_\tau$, the \emph{cone rooted} at $b_\tau$ is the set of nodes
		\begin{align*}
			\{b_{\tau^\prime}:\tau^\prime\in{^{<n}2}\textrm{ and }\tau^\prime\textrm{ extends }\tau\}.
		\end{align*}
	\end{definition}
	
	\begin{definition}\cite[Appendix A]{NA22}
		Let $T$ be an $n$-tree for $\varphi$. The notion of $m$\emph{-subtree} of $T$ is defined by induction on $m\geq 1$ as follows. The $1$-subtrees are the leaves of $T$. An $m$-subtree is a set of nodes of $T$ obtained from an arbitrary internal node $b_\tau$, together with an $(m-1)$-subtree contained in the cone rooted at $b_{\tau^\smallfrown 0}$ and an $(m-1)$-subtree contained in the cone rooted at $b_{\tau^\smallfrown 1}$. The \emph{root} of such subtree is the node $b_\tau$. The root of a $1$-subtree is the only node it contains.
	\end{definition}
	
	In the next two results we show that every subtree can be transformed into a tree adding to it some branches of the original tree.
	
	\begin{lemma}\label{difleaves}
		Let $T$ be an $n$-tree for $\varphi$ and let $H$ be an $m$-subtree. Given two different leaves $b_\sigma$ and $b_{\sigma^\prime}$ of $T$ which lie on $H$, define $l_0=\max\{l:\sigma\vert l=\sigma^\prime\vert l\}$. Then $l_0<n-1$ and $b_{\sigma\vert l_0}=b_{\sigma^\prime\vert l_0}\in H$.
	\end{lemma}
	
	\begin{proof}
		Since $b_\sigma\neq b_{\sigma^\prime}$, then $l_0<n-1$. The other part of the statement is proved by induction on $m\geq 2$. For $m=2$ it follows by definition of subtree. Suppose that the result is true for every $(m-1)$-subtree of $T$, and let $H$ be an $m$-subtree whose root is a node $b_\tau$. If $\sigma=\tau^\smallfrown 0\cdots$ and $\sigma^\prime=\tau^\smallfrown 1\cdots$ (or vice versa), then we have that $b_{\sigma\vert l_0}=b_{\sigma^\prime\vert l_0}=b_\tau\in H$. Otherwise, both $b_\sigma$ and $b_{\sigma^\prime}$ lie on an $(m-1)$-subtree $H^\prime$ contained in the cone rooted at, say, $b_{\tau^\smallfrown 0}$. By induction hypothesis, $b_{\sigma\vert l_0}=b_{\sigma^\prime\vert l_0}\in H^\prime\subset H$, as required.
	\end{proof}
	
	\begin{lemma}\label{subtree}
		Let $T$ be an $n$-tree for $\varphi$, and let $H$ be an $m$-subtree. Then $T$ contains branches $(c_\sigma:\sigma\in{^m}2)$ such that the nodes of $H$, together with these branches, form an $m$-tree for $\varphi$.
	\end{lemma}
	
	\begin{proof}
		Let $b_\sigma\in H$ be a leaf of $T$. By definition of subtree it follows that there are $i_0,\ldots,i_{m-1}$ such that $b_{\sigma\vert i_0},\ldots,b_{\sigma\vert i_{m-1}}\in H$. Observe that $i_{m-1}=n-1$, so $b_{\sigma\vert i_{m-1}}=b_\sigma$. Now define $\sigma_0,\sigma_1\in{^m2}$ as follows:
		\begin{align*}
			&\sigma_0=(\sigma(i_0),\ldots,\sigma(i_{m-2}),0),\\
			&\sigma_1=(\sigma(i_0),\ldots,\sigma(i_{m-2}),1).
		\end{align*}
		Let $c_{\sigma_0}=a_{\sigma^\smallfrown 0}$ and $c_{\sigma_1}=a_{\sigma^\smallfrown 1}$, which are branches of $T$. These branches correspond with the nodes $d_{\sigma_0\vert j}=d_{\sigma_1\vert j}=b_{\sigma\vert i_j}$, for $j=0,\ldots,m-1$. Then, for $k=0,1$ and $j=0,\ldots,m-1$, we have
		\begin{align*}
			\varphi(c_{\sigma_k},d_{\sigma_k\vert j})\Leftrightarrow\varphi(a_{\sigma^\smallfrown k},b_{\sigma\vert i_j})\Leftrightarrow\sigma^\smallfrown k(i_j)=0\Leftrightarrow\sigma_k(j)=0.
		\end{align*}
		
		To conclude, we must prove that these branches are well-defined, that is, different leaves induce different branches. In this way, the $2^{m-1}$ leaves of $T$ that lie on $H$ induce $2^{m}$ different branches, as desired. Given two different leaves $b_\sigma$ and $b_\tau$ of $T$ in $H$, define $l_0=\max\{l:\sigma\vert l=\tau\vert l\}$. It follows by Lemma \ref{difleaves} that $l_0<n-1$ and $b_{\sigma\vert l_0}=b_{\tau\vert l_0}\in H$. Hence, for $k=0,1$ we have
		\begin{align*}
			&\sigma_k=(\sigma(i_0),\ldots,\sigma(l_0),\ldots,k),\\
			&\tau_k=(\sigma(i_0),\ldots,\tau(l_0),\ldots,k),
		\end{align*}
		and $\sigma(l_0)\neq\tau(l_0)$. So the branches $\sigma_0$ and $\sigma_1$ are both different from the branches $\tau_0$ and $\tau_1$, as we wanted to prove.
	\end{proof}
	
	Henceforth, the term \emph{subtree} will not denote merely a set of nodes, but rather the actual tree obtained as in Lemma \ref{subtree}.
	
	\begin{notation}
		Let $T$ be an $n$-tree for $\varphi$, and for $j=0,1$ let $T_j$ be the cone rooted at the node $b_j$, which is an $(n-1)$-subtree of $T$.
	\end{notation}
	
	\begin{lemma}\label{ramseytree}
		Let $p$ and $q$ be non-negative integers with $p+q\geq 1$, and let $T$ be a $(p+q)$-tree for $\varphi$. If the nodes of $T$ are partitioned into two sets $P$ and $Q$, then either $P$ contains $p$-subtree or $Q$ contains a $q$-subtree.
	\end{lemma}
	
	\begin{proof}
		By induction on $p+q\geq 1$. The case $p+q=1$ is trivial, since $b_\emptyset$ is either in $P$ or in $Q$. Suppose that $p+q>1$, and assume without loss of generality that $b_\emptyset\in P$. For $j=0,1$ consider the $(p+q-1)$-subtree $T_j$ of $T$. By induction hypothesis, either $P\cap T_j$ contains a $(p-1)$-subtree or $(Q\cap T_j)$ contains a $q$-subtree. If at least one of $Q\cap T_0$ and $Q\cap T_1$ contains a $q$-subtree, then so does $Q$. Otherwise, both $P\cap T_0$ and $P\cap T_1$ contain a $(p-1)$-subtree. These subtrees, together with $b_\emptyset$, which is in $P$, yields the existence of a $p$-subtree in $P$.
	\end{proof}
	
	Finally, we reproduce Hodges's proof of \cite[Lemma 6.7.9]{WH93} using the previous results.
	
	\begin{theorem}\label{tree2}
		Let $T$ be a tree for $\varphi$ of height $2^{n+1}-2$. Then $T$ has branches $a_0,\ldots,a_{n-1}$ and nodes $b_0,\ldots,b_{n-1}$ such that $\varphi(a_i,b_j)$ holds if and only if $i\leq j$.
	\end{theorem}
	
	\begin{proof}
		We will prove by induction on $n-r=0,\ldots,n-1$ that, for $1\leq r\leq n$, the following situation $S_r$ holds: for some $k=0,\ldots,n-r$ there are
		\begin{align*}
			a_0,b_0,\ldots,a_{k-1},b_{k-1},H,a_k,b_k,\ldots,a_{n-r-1},b_{n-r-1}
		\end{align*}
		such that:
		\begin{enumerate}
			\item[$(1)$] $H$ is a subtree of $T$ of height $2^{r+1}-2$,
			\item[$(2)$] for all $i,j<n-r$, the formula $\varphi(a_i,b_j)$ holds if and only if $i\leq j$,
			\item[$(3)$] if $b$ is a node of $H$, then $\varphi(a_i,b)$ holds if and only if $i<k$,
			\item[$(4)$] if $a$ is a branch of $H$, then $\varphi(a,b_j)$ holds if and only if $j\geq k$.
		\end{enumerate}
		The initial case $S_n$ states simply states that $H$ is a tree of height $2^{n+1}-2$ for $\varphi$, which is true for $H=T$. The statement of the theorem is implied by the case $S_1$ as follows. The existence of $H$, which is a $2$-subtree of $T$, together with Lemma \ref{subtree}, yields the existence of a branch $a$ and a node $b$, both of them of $T$ but in particular of $H$, such that $\varphi(a,b)$ holds. Hence, the sequence we are looking for is
		\begin{align*}
			a_0,b_0,\ldots,a_{k-1},b_{k-1},a,b,a_k,b_k,\ldots,a_{n-2},b_{n-2}.
		\end{align*}
		It remains to show that if $S_r$ holds, then so does $S_{r-1}$. Let $h=2^r-2$, so that $H$ is a tree of height $2h+2$ for $\varphi$. Recall that the nodes and branches of $H$ are nodes and branches of $T$. Now for each branch $a$ of $H$ write $H(a)$ for the set of nodes b of $H$ such that $\varphi(a,b)$ holds. There are two cases.
		
		Assume first that there is a branch $a$ of $H$ such that $H(a)$ contains an $(h+1)$-subtree whose root is some $b_\tau$. Then $H$ contains a node $b=b_\tau$ and an $h$-subtree $H^\prime$ contained in the cone rooted at $b_{\tau^\smallfrown 1}$. Replace $H$ by $a,b,H^\prime$ in that order.
		
		Assume now that for every branch $a$ of $H$, the set $H(a)$ contains no $(h+1)$-subtree. Let $b=b_\tau$ be the root of $H$, and let $a$ be any branch of $H_0$, which is a $(2h+1)$-tree for $\varphi$ rooted at a node of the form $b_{\tau^\smallfrown0\cdots}$. Since $H_0$ is a tree of height $(h+1)+h$ and $H_0\cap H(a)$ contains no $(h+1)$-subtree, then by Lemma \ref{ramseytree}, the set $H_0\setminus H(a)$ contains an $h$-subtree $H^\prime$. Replace $H$ by $H^\prime,a,b$ in that order.
		
		In either case, $S_{r-1}$ holds.
	\end{proof}

\end{document}